\documentclass[none,preprint]{imsart}

\RequirePackage[OT1]{fontenc}
\RequirePackage{amsthm,amsmath,bbm,amssymb}
\RequirePackage{natbib}
\RequirePackage{booktabs,caption}
\RequirePackage[colorlinks,citecolor=blue,urlcolor=blue]{hyperref}

\newtheorem{definition}{Definition}
\newtheorem{lemma}{Lemma}
\newtheorem{proposition}{Proposition}
\newtheorem{theorem}{Theorem}
\newtheorem{fact}{Fact}

\theoremstyle{remark}
\newtheorem{remark}{Remark}

\theoremstyle{remark}

\theoremstyle{example}
\newtheorem{example}{Example}

\theoremstyle{corollary}
\newtheorem{corollary}{Corollary}

\startlocaldefs
\numberwithin{equation}{section}
\theoremstyle{plain}

\endlocaldefs

\renewcommand{\tablename}{Display}

\begin{document}

\begin{frontmatter}
\title{Nonparametric clustering of functional data using pseudo-densities}
% \thanksref{T1}}
\runtitle{Nonparametric clustering of functional data}
% \thankstext{T1}{Footnote to the title with the ``thankstext'' command.}

\begin{aug}
\author{\fnms{Mattia} \snm{Ciollaro}\ead[label=e1]{ciollaro@cmu.edu}},
\author{\fnms{Christopher R.} \snm{Genovese}\thanksref{t3}\ead[label=e2]{genovese@stat.cmu.edu}}
\and
\author{\fnms{Daren} \snm{Wang}\ead[label=e3]{darenw@andrew.cmu.edu}}
% \ead[label=u1,url]{http://www.foo.com}}

% \thankstext{t1}{Some comment}
% \thankstext{t2}{First supporter of the project}
\thankstext{t3}{Research supported in part by National Science Foundation Grants
  NSF-DMS-1208354 and NFS-DMS-1513412.}
\runauthor{Ciollaro, Genovese, and Wang}

\affiliation{Carnegie Mellon University}
 % and Another University\thanksmark{m2}}

\address{Department of Statistics\\
Carnegie Mellon University\\
5000 Forbes Avenue\\
Pittsburgh, PA 15213\\
\printead{e1}\\
\phantom{E-mail:\ }\printead*{e2}\\
\phantom{E-mail:\ }\printead*{e3}}

\end{aug}

\begin{abstract}
We study nonparametric clustering of smooth random curves on the basis of the
$L^2$ gradient flow associated to a pseudo-density functional and we show that 
the clustering is well-defined both at the population and at the sample level. 
%We study nonparametric clustering of smooth random curves on the basis of the
%critical points of a pseudo-density functional. We show that the population clusters
%corresponding to the basins of attraction of the critical points of such
%pseudo-density are well-defined in terms of the $L^2$ gradient flow paths.
% associated to the pseudo-density. 
% We show that the pseudo-density and its first
% two derivatives can be consistently estimated from noisy samples of discretized
% curves.
We provide an algorithm to mark significant local modes, which are
associated to informative sample clusters, and we derive its consistency properties.
% that are likely associated to
%local modes of the population pseudo-density. 
%The proposed algorithm also outputs
%confidence regions for the location of the population local modes associated to
%these clusters.
Our theory is developed under weak assumptions, which essentially
reduce to the integrability of the random curves, and does not require to project 
the random curves on a finite-dimensional
subspace. 
However, if the underlying probability distribution
is supported on a finite-dimensional subspace, we show that the
pseudo-density and the expectation of a kernel density estimator induce the
same gradient flow, and therefore the same clustering.
Although our theory is developed for smooth curves that belong to an 
infinite-dimensional functional space, we also provide consistent procedures that can be
used with real data (discretized and noisy observations).
%population clusters corresponding to the $L^2$ gradient flow associated to the
%pseudo-density and the population clusters induced by the expectation of a
%kernel density estimator of the density are equivalent. 
%In this sense, the approach to clustering that we describe is adaptive to the
%intrinsic dimension of the sample space.
\end{abstract}

% \begin{keyword}[class=MSC]
% \kwd[Primary ]{60K35}
% \kwd{60K35}
% \kwd[; secondary ]{60K35}
% \end{keyword}


\end{frontmatter}

\section{Introduction}

In Functional Data Analysis (\citealp{ramsay2005functional},
\citealp{ferraty2006nonparametric}, \citealp{ferraty2011oxford},
\citealp{horvath2012inference}), henceforth FDA, we think of curves (and other
  functions) as the fundamental unit of measurement. Clustering is an
  important problem in FDA because it is often of critical interest to
  identify subpopulations based on the shapes of the measured curves. In
  this paper, we study the problem of functional clustering in a fully
  infinite-dimensional setting. We are motivated by recent work on modal
  clustering in finite dimensions (\citealp{chacon2012clusters},
  \citealp{chacon2014populationbackground}, and references therein) that, in
  contrast to many commonly-used clustering methods, has a population
  formulation, and by recent advances in clustering of functional data (\citealp{bongiorno2015clustering}). Specifically, we prove the existence of population
  clusters in the infinite-dimensional functional case, under mild
  conditions. We show that an analogue of the mean-shift algorithm (see, for example,
\citealp{fukunagahostetlermeanshift},
\citealp{chengmeanshift}, and the more recent works of
\citealp{comaniciuadaptivemeanshift} and \citealp{carreira2006fast})
  can identify local modes of a ``pseudo-density''. We
  devise an algorithm to classify local modes as representatives of
  significant clusters, and under some regularity assumptions on the pseudo-density, we further show
  that the algorithm is consistent. We develop our theory assuming that the data are observed as continuous curves defined on some interval. 
  Because in practice one does not observe continuous curves, we also show how to apply the procedures that we propose to
  real data (e.g. noisy measurements of random curves on a grid).

Modal clustering is typically a finite-dimensional problem, but motivated by the flourishing literature on FDA and by the increasing interest in developing sound
frameworks and algorithms for clustering of random curves,
%and by the increasing number of applications involving
%functional data, 
we extend the idea of modal clustering to the
case where $X$ is a functional random variable valued in an
infinite-dimensional space. In particular,
we develop a theory of modal clustering for smooth random curves that are assumed to belong to the
H\"{o}lder space $H^1([0,1])$
% In this paper, $\mathcal{X}=H^1([0,1])$ denotes
% the space 
of curves defined on the standard unit interval whose first weak derivative is square
integrable. We focus on $H^1([0,1])$  for concreteness, but our theory generalizes to any
H\"{o}lder space.
Furthermore, our theory is
density-free and nonparametric, as no assumptions are made regarding the
existence of a dominating measure for the law $P$ of the functional data, nor it
is assumed that $P$ can be parametrized by a finite number of parameters.

In the finite-dimensional modal clustering problem, we have that $p: \mathcal{X} \to \mathbb{R}_+$ is the probability density function associated to the
law $P$ of a random variable $X$ valued in $\mathcal{X} \subseteq \mathbb{R}^d$. If
$p$ is a Morse function (i.e. $p$ is smooth and its Hessian is not singular at
the critical points), then the local modes of $p$,
$\mu_1, \dots, \mu_k$, induce an 
% essential 
partition of the sample space 
% $\mathcal{X}$. In particular, $\mathcal{X}$ can be partitioned as
$\mathcal{X} = C_1 \cup C_2 \cup \dots \cup C_k$ where the sets $C_i$ satisfy
\begin{enumerate}
\item $P(C_i)>0$ $\forall i=1,\dots,k$
\item $P(C_i\cap C_j) = 0$ if $i \neq j$
\item $P(\cup_{i=1}^k C_i)=1$
\item $x \in C_i \iff $ the gradient ascent path on $p$ that starts from $x$
  eventually converges to $\mu_i$.
\end{enumerate} 
%
% (see \citealp{chacon2012clusters} and \citealp{chacon2014populationbackground}).
% Furthermore, each point $x \in C_i$ is
% such that the gradient ascent path starting from $x$
% converges to $\mu_i$. 
% Conditions 1, 2, and 3 motivate the name \textit{essential} partition while,
% intuitively, 
Note that this framework characterizes $C_i$ as a high-density region surrounding the local
mode $\mu_i$ of $p$ and each set $C_i \in \mathcal{C}$
% motivates why this often referred to as `modal clustering' because each set $C_i \in
% \mathcal{C}$ 
is thought of as a cluster at the population
level.
% Furthermore, the essential partion $\mathcal{C}$ of $\mathcal{X}$
% captures in a very precise way the notion of population clustering. 
Unlike other approaches to clustering which define clusters exclusively at the sample level (consider $k$-means, for instance),
modal clustering provides an inferential framework in which the essential partition
$\mathcal{C}$ is a population parameter that one wants to infer from the data. In fact, as soon as an i.i.d. sample $X_1,
\dots, X_n \sim P$ and an estimator $\hat p$ of $p$ are available,
the goals of modal clustering are exactly
\begin{itemize}
\item estimating the local modes of $p$ by means of the local modes of $\hat p$
\item estimating the population clustering $\mathcal{C}=\{C_1,
\dots,C_k\}$ by means of the empirical partition $\hat{\mathcal{C}}=\{\hat C_1,
\dots, \hat C_{\hat k}\}$
induced by $\hat p$
\end{itemize}
Thus, the typical output of a modal clustering procedure consists of the estimated
clustering structure $\hat{\mathcal{C}}$ and a set of cluster
representatives $\hat \mu_1, \dots, \hat \mu_{\hat{k}}$. At the sample level, each data point is then
uniquely assigned to a cluster $\hat C_i \in \hat{\mathcal{C}}$ and represented by the
corresponding local mode $\hat \mu_i$.

Because it is generally not possible to define a probability density function in
infinite-dimensional Hilbert spaces, we instead focus on a
surrogate notion of density which we call ``pseudo-density''.
Generally, by pseudo-density we mean
any suitably smooth functional which maps the sample space $\mathcal{X}$ into the positive reals
$\mathbb{R}_+=[0,\infty)$. In particular, we focus on a
 family of pseudo-densities $\mathcal{P}=\{ p_h :
\mathcal{X} \to \mathbb{R}_+; h>0 \}$ which is parametrized by a bandwidth parameter
$h$ and, more specifically,
% In this setting, it is impossible to define a probability density 
% to describe the data generating process. We thus take 
$p_h$ is the expected value of a kernel density
estimator,
\begin{equation}
p_h(x)=E_P \, K\left(
  \frac{\|X-x\|^2_{L^2}}{h}  \right), 
\end{equation}
where $K$ is
an appropriately chosen kernel function and $h$ is the bandwidth
parameter. 
% However, no
% assumptions are made regarding the existence of a dominating
% measure for $P$ or the existence of a density.
 Clusters of curves are then defined in terms of the $L^2$ gradient flow associated
to $p_h$.

% The theory of modal clustering that
% we develop is based on the $L^2$ gradient flow associated to a pseudo-density.
% which can substitute the probability density function associated
% with the data when it does not exist. 

The gradient flow associated to $p_h \in \mathcal{P}$ is the collection
of the gradient ascent paths $\pi_x : \mathbb{R}_+ \to \mathcal{X}$ corresponding to
the solution of the initial value problem
\begin{equation}
  \label{eq:ivp}
  \begin{cases}
    \frac{d}{dt} \pi_x(t) = \nabla p_h(\pi_x(t)) \\
    \pi_x(0) = x,
  \end{cases}
\end{equation}
where $\nabla p_h(x)$ is the $L^2$ functional gradient of $p_h$ at $x \in
\mathcal{X}$. In complete analogy with the finite-dimensional case, the gradient
of $p_h$ induces a vector field and a gradient ascent path is a curve 
$\pi_x \subset H^1([0,1])$ that solves the initial value problem and moves along 
the direction of the vector field, i.e. at any time $t \geq 0$, the derivative
of $\pi_x(t)$ corresponds to the gradient of $p_h$ at $\pi_x(t)$.
%  with the direction of steepest ascent of the functional
% $p_h$ (in our case, steepest ascent is intended with respect to the $L^2$
% norm). 
If the trajectory $\pi_x$ converges to a local mode $\mu_i$ of $p_h$ as
$t \to \infty$, then $x$ is said to belong to the $i$-th cluster of $p_h$,
$C_i$. Thus, the cluster $C_i$ is defined as the set
\begin{equation}
\label{eq:populationCluster}
C_i = \left\{ x \in H^1([0,1]) : \lim_{t \to \infty} \|\pi_x(t)-\mu_i\|_{L^2([0,1])} \to 0 \right\}
\end{equation}
where $\pi_x$ is a solution of the initial value problem of equation
\eqref{eq:ivp}. According to the this definition, the $i$-th cluster of $p_h$
corresponds to the basin of attraction of the $i$-th local mode $\mu_i$ of
$p_h$, and the collection of the clusters $C_i$ provides a good summary of the
subpopulations associated to $(\mathcal{X},P)$.

The main contribution of our work is to identify conditions under which 
\begin{enumerate}
\item there exist population clusters in functional data, i.e. the population clusters defined in
equation \eqref{eq:populationCluster} exist and are well-defined
\item these clusters are estimable
\end{enumerate}
and to provide a practical procedure to estimate the clusters and assess their statistical significance.
%
% one can
% prove the existence of population clusters in functional data and estimate
% them. 
% This
% corresponds to comprehending under which conditions the clusters defined in
% equation \eqref{eq:populationCluster} exist and are well-defined. 
%In order to
%achieve this goal, one needs to carefully study the initial value problem of equation
%\eqref{eq:ivp} and understand the conditions needed to uniquely solve it for any
%initial value $x=\pi_x(0)$ in such a way that the corresponding solution trajectory $\pi_x$
%converges to a critical point of $p_h$ as $t \to \infty$.

As we further discuss later in the paper, the most remarkable challenge arising in
the infinite-dimensional setting is the lack of compactness. As opposed to the
finite-dimensional setting, in the functional case it is
hard to show the existence, the uniqueness, and the
convergence of the gradient ascent paths described by the initial value problem
of equation \eqref{eq:ivp}, unless the sample
space $\mathcal{X}$ can be compactly embedded in another space.
% Needless to say, if one is not able to establish the existence,
% the uniqueness, and the convergence of such paths to the critical points of the
% pseudo-density $p_h$, then the development of a statistical theory
% of clustering based on pseudo-densities becomes hopeless. 
We show that we can overcome this challenge by exploiting the compact embedding of $H^1([0,1])$ in $L^2([0,1])$, the
space of square-integrable functions on the unit interval, and by studying
equation \eqref{eq:ivp} using these two non-equivalent topologies. For convenience, we focus on functional data belonging to
$H^1([0,1])$ and on the gradient flow under the $L^2$ norm, but the
exact same theory carries over to other function spaces, different
norms and different pseudo-density functionals, as long as it is
possible to compactly embed the sample space $\mathcal{X}$
in a larger space and the chosen pseudo-density functional is
sufficiently smooth. In particular, we remark that the results of this paper can
be straightforwardly
generalized to arbitrary pairs of Sobolev spaces of integer order satisfying the compact embedding requirement.

%Another challenge
%is related to saddle points. In the finite-dimensional setting,
%clusters associated to saddle points of a Morse density $p$ are negligible
%because they have null probability content (in fact, they correspond to
%manifolds of dimension strictly lower than that of the domain of $p$). However, in the infinite-dimensional setting
%that we consider (which is density-free), the population clusters associated to
%the saddle points of $p_h$ are not necessarily negligible in terms of
%probability content even if $p_h$ is a Morse functional. This
%fact strongly motivates the need of a procedure to classify a local mode of an
%estimate $\hat p_h$ as a significant mode if it is located close to a local mode of
%$p_h$. We provide an algorithm to accomplish this task which is based on the
%curvature of $\hat p_h$ at its critical points. If the curvature of $\hat p_h$ is sufficiently
%negative at a critical point $x^*$, then $x^*$ is marked as a
%relevant local mode of $\hat p_h$ and the algorithm outputs a
%confidence region for the local mode of $p_h$ that is likely located next to $x^*$.
%% For this reason, at
%% least in this paper, we focus on `critical-points-based clustering', rather than
%% on the `modal clustering' theory of \cite{chacon2012clusters} and \cite{chacon2014populationbackground}.

The theory of clustering that we develop in this work is projection-free, since it does not involve projecting the
random curves onto a finite-dimensional space.
% , and it does not reduce the
% problem to a finite-dimensional clustering problem. 
However, if the probability law $P$ of the functional data is supported on a finite-dimensional
space and admits a proper density with respect to the Lebesgue measure, we show that the
gradient flow on the pseudo-density $p_h$ and the gradient flow on the
expectation of the kernel density estimator of the data coincide (and so coincide the corresponding population clusterings.

One of the most important practical tasks in modal clustering is to identify significant local modes, as these are
associated to informative clusters. We provide an algorithm that
\begin{itemize}
\item identifies the local modes of the population pseudo-density $p_h$ by analyzing its sample version $\hat p_h$; 
furthermore, all of the local modes of $\hat p_h$ identified by the algorithm  converge asymptotically to their population 
correspondents of $p_h$
\item is consistent (under additional regularity assumptions on $p_h$), in the sense that it establishes a one-to-one
correspondence between the sample local modes that it identifies and their population equivalents.
\end{itemize}

While from a purely mathematical standpoint a sample of functional data $\{X_i\}_{i=1}^n$ is thought of as
a collection of continuous curves defined on an interval, we never observe such objects in practice. Rather, we typically only
observe noisy measurements of the $X_i$'s at a set of design points $\{t_j\}_{j=1}^m$. As an intermediate step, we therefore estimate the $X_i$'s from 
these observations (which constitutes a typical regression problem), and then use the estimates as the input of our procedure. 
%The consistency 
%of the algorithm which identifies the significant local modes follows by well-known properties of these estimates.

The remainder of the paper is organized as follows. Section
\ref{section:relatedLiterature} provides a concise literature review on
pseudo-densities, finite-dimensional nonparametric clustering
based on Morse densities, and mode-finding algorithms. Section
\ref{section:populationBackground} is devoted to the development of our theory of population
clustering for smooth random curves. In particular, Section
\ref{section:populationBackground} studies in detail the $L^2$ gradient flow on
the pseudo-density $p_h$ and establishes that, in analogy to the
finite-dimensional case, population clusters of smooth random curves can be
defined in terms of the basins of attraction of the critical points of
$p_h$. Section \ref{section:weakAdaptivity} describes the behavior of
the $L^2$ gradient flow of $p_h$ when the probability law $P$ of the data is
supported on a finite-dimensional subspace. Section
\ref{section:confidenceMarking} provides an algorithm to identify the significant
local modes of $\hat p_h$ and shows that, under additional regularity assumptions on $p_h$, the algorithm is consistent.
% Furthermore, the
% algorithm outputs a confidence ball in the $L^2$ norm for the location of
% the local modes of $p_h$ that are presumably associated to the relevant local
% modes of $\hat p_h$. 
% Section \ref{section:inference} provides exponential
% probability bounds for the distance of $\hat p_h$ and its derivatives to $p_h$
% and its derivatives. 
Section \ref{sec:theoryToApplications} extends the results of Section
\ref{section:confidenceMarking} to real data.
Section \ref{section:choiceOfPseudoDensity} contains a discussion on the choice of the pseudo-density functional and some general
guidelines for the choice of the smoothing parameter $h$ in practical
applications. Section \ref{section:conclusionAndDiscussion} summarizes the main
contributions of this paper and indicates promising directions for future work. The proofs of the main results can be
found in Appendix \ref{section:appendix}, while other 
auxiliary results (such as probability bounds for the estimation of the
pseudo-density functional $\hat p_h$ and its derivatives)
are deferred to Appendix \ref{section:additionalResults}.

\section{Related literature}
\label{section:relatedLiterature}
The difficulties associated to the lack of proper density functions
in infinite-dimensional spaces are well-known among statisticians.
This has stimulated the introduction of various surrogate notions of density for
functional spaces. The
literature on pseudo-densities includes the work of
\cite{gasser1998nonparametric}, \cite{hall2002depicting},
\cite{dabo2004estimation}, \cite{delaigle2010}, and
\cite{ferratysurrogatedensity}. 

A population framework based on Morse theory for nonparametric modal clustering in the finite
dimensional setting is presented in \cite{chacon2012clusters} and
\cite{chacon2014populationbackground}. Whenever a proper density $p:\mathcal{X} \to
\mathbb{R}^d$ exists and it is a Morse function, the
problem of equation \eqref{eq:ivp} induces an essential partition of
the sample space $\mathcal{X} \subseteq R^d$ in the sense that each set $C_i$ in the
partition of $\mathcal{X}$ such that $P(C_i)>0$ corresponds to the basin of attraction of a
local mode $\mu_i$ of $p$, i.e. $C_i=\{ x \in \mathcal{X} : \lim_{t
  \to \infty} \pi_x(t) = \mu_i \}$. Furthermore, if $p$ has saddle points, the basin of
attraction of each saddle is a null probability set (similarly, the basin of
attraction of a local minimum is a singleton and hence negligible as well). 

A number of gradient ascent
algorithms have been developed to perform modal clustering in the
finite-dimensional case. One of the most popular mode-finding and
modal clustering algorithms is the \textit{mean-shift algorithm}
(\citealp{fukunagahostetlermeanshift},
\citealp{chengmeanshift}). A version of the mean-shift algorithm for functional data is
discussed in \cite{ciollaromeanshift}. A gradient ascent
algorithm for functional data is proposed in
\cite{hall2002depicting}. 

In their recent work,
\cite{bongiorno2015clustering} propose a clustering method for functional data
based on the small ball probability function $\varphi_h(x)=P(\|X-x\|^2\leq h)$ and on functional principal
components. A recent overview of other clustering techniques for functional data
can be found in \cite{jacques2013functional}. 

%Just as
%\cite{chacon2012clusters} and \cite{chacon2014populationbackground} set the
%basis for an inferential theory of modal clustering in the finite
%dimensional case, our work establishes a general
%population background for modal clustering
%of functional data. 
% To a
% large extent, it is exactly removing the assumption of a finite-dimensional structure for the
% data that makes the problem particularly interesting from the point of view of the theory of functional
% data analysis.

\section{A population background for pseudo-density clustering of functional data}
\label{section:populationBackground}
We denote by $X \sim P$ a functional random variable valued in $L^2([0,1])$, the
space of square integrable functions on the unit interval with its
canonical inner product $\langle x,y \rangle_{L^2} = \int_0^1
x(s)y(s)\, ds$ and induced norm $\|x\|_{L^2}=\sqrt{\langle x,x
  \rangle_{L^2}}$. As we previously mentioned, it is not possible to
define a proper probability density function for $P$. Instead, we study
the $L^2$ gradient flow of equation \eqref{eq:ivp} associated to the functional
\begin{equation}
\label{eq:ph}
p_h(x)=E_P \, K \left( \frac{\|X-x\|^2_{L^2}}{h} \right) =
\int_\mathbb{R} K (s)\,d P_{\|X-x\|^2_{L^2}/h}(s)
\end{equation}
mapping $L^2([0,1])$ into $\mathbb{R}_+$, where $h>0$ is a bandwidth parameter, $K: \mathbb{R}_+ \to
\mathbb{R}_+$ is a kernel function, and $P_{\|X-x\|^2_{L^2}/h}$
denotes the probability measure induced by $P$ through the map $X
\mapsto \|X-x\|^2_{L^2}/h$. Note that $p_h$ is closely related to the so-called \emph{small-ball
probability} function $\varphi_h(x)=P(\|X-x\|^2_{L^2}\leq h)$. It is easy to see that $p_h(x)=\varphi_h(x)$ when
$K(s)=\mathbbm{1}_{[0,1]}(s)$, therefore $p_h$ can be thought of as a smoother version of $\varphi_h$.

Unless otherwise noted, we make the following assumptions throughout the paper:
\begin{itemize}
% \item[(H1)] $K_h: \mathbb R \to \mathbb R $  is twice continuously
%   differentiable. Further, for both $K_h^l(t)$ and  $K_h^l(t^2) t^l$ are a uniformly
% bounded functions on $\mathbb R$, say, $$\sup_{t\in \mathbb R}\Big\{|K_h^l (t)|+ |K_h^l (t^2) t^l | \Big\}\le K_l < \infty.$$ This is satisfied if $K$ is compactly supported or $K(t)=e^{-t}$.
% \item[(H2)] $K'(t)+K(t) \leq 0$ for all $t \in \mathbb{R}_+$.
% \item[(H3)] $X$ is $P$-almost surely absolutely continuous and the moments
%   $E_P\|X\|_{L^2}^l \le M_l$, and that $E_P\|X'\|_{L^2}^l \le N_l$ are finite. 
% \item[(H4)] All the critical points of $p_h$ are isolated under the $L^2$
  % norm.
\item[(H1)] $K: \mathbb R_+ \to \mathbb R_+ $  is twice continuously
  differentiable and the following bounds hold on the derivatives of $K_h(\cdot)=K(\cdot/h)$:
  % $K: \mathbb R_+ \to \mathbb R_+ $  is twice continuously
  % differentiable and its third derivative is bounded. Furthermore, we assume
  % that, for $\ell \in \{0,1,2,3\}$, the $\ell$-th derivative of $K_h(t^2)$ is
  % uniformly bounded. This assumption can be summarized by the following three
  % statements:
  %
  \begin{itemize}
  \item $\sup_{t \in \mathbb{R}_+} \left| K_h(t^2) \right| \leq K_0 < \infty$
  \item $\sup_{t \in \mathbb{R}_+} \left| K'_h(t^2)t \right| \leq K_1 < \infty$
  \item $\sup_{t \in \mathbb{R}_+} \left\{\left| K^{(\ell-1)}_h(t^2)t^{\ell-2}
        \right| + \left|K_h^{(\ell)}(t^2)t^\ell\right| \right\} \leq K_\ell < \infty$, for $\ell=2,3$
  \end{itemize}
  where the constants $K_0, K_1, K_\ell$ may depend on $h$.
\item[(H2)] $K'(t^2)+K(t^2) \leq 0$ for all $t \in \mathbb{R}_+$.
\item[(H3)] $X$ is $P$-almost surely absolutely continuous and its moments satisfy
  $E_P\|X\|_{L^2} \leq M_1<\infty$ and $E_P\|X'\|_{L^2} \leq
  N_1<\infty$ for some constants $M_1$ and $N_1$.
\item[(H4)] All the non-trivial critical points of $p_h$ are isolated under the $L^2$
  norm, i.e. there exists an open $L^2$ neighborhood around each critical point $x^*$
  of $p_h$ with $p_h(x^*)>0$ such that there are no other critical points of
  $p_h$ that also belong to that neighborhood.
\end{itemize}
Various kernels can be shown to satisfy assumptions (H1) and (H2). For instance,
both the compactly supported kernel $K(t) \propto (1-t)^3\mathbbm{1}_{[0,1]}(t)$ and the
exponential kernel $K(t) \propto e^{-t}\mathbbm{1}_{[0,\infty)}(t)$ satisfy our
assumptions. (H3) is an assumption on the smoothness of the random
curves. Intuitively, (H3) corresponds to assuming that the probablility law $P$
does not favor curves that are too irregular or wiggly. (H4) is a regularity
assumption on the functional $p_h$: essentially, under the above assumptions on $K$, (H4) corresponds to assuming
that the functional $p_h$ does not have flat ``ridges'' in regions where it is positive.
\begin{remark}
\label{remark:commentsOnMorseCondition}
A sufficient condition for (H4) to hold is that $p_h$ is a Morse
functional. The following Proposition provides a sufficient condition under which
$p_h$ is a Morse functional.
% Furthermore, if the distribution of $X$ is supported on a compact domain $S_c$ that is a subset of a
% finite-dimensional space $S \subset L^2$ and if $P$ admits a
% proper Morse density $p$ with respect to the Lebesgue measure, one can show
% that under mild conditions there exists $h_0 >0$ such that if $0 < h \leq h_0$ then $p_h$ is a
% Morse function on the interior of that domain. 
% In the next Proposition, we
% provide a set of sufficient conditions
% under which $p_h$ is Morse in the finite-dimensional setting. Furthermore, as we will see in
% Lemma \ref{} of Section \ref{}, this Proposition allows us to establish that $p_h$ is a
% Morse functional (and thus satisfies condition H4) if 
\end{remark}
\begin{proposition}
\label{proposition:phIsMorse}
Suppose that $P$ has density $p$ with respect to the Lebesgue measure and $p$ is
supported on a finite-dimensional compact domain $S_c \subset
\mathbb{R}^d$. Suppose furthermore that $p$ and $\partial S_c$, the boundary of $S_c$, satisfy
\begin{itemize}
\item $\partial S_c$ is smooth enough so that the normal
  vector $n(x)$ exists for any $x \in \partial S_c$
\item $p$ is continuous on $\mathbb{R}^d$
\item $p$ is twice differentiable in the interior of $S_c$,
  $\operatorname{int}(S_c)$
\item $\nabla p$ is not vanishing on $\partial S_c$.
\end{itemize}
%
% or that $p$ and $\partial S_c$ satisfy
% %
% \begin{itemize}
% \item $\partial S_c$ is smooth enough so that $n(\cdot)$ exists and $n(\cdot)$
%   is a Lipschitz map
% \item $p$ is twice differentiable in $\operatorname{int}(S_c)$
% \item $\langle \nabla p(x), n(x)\rangle_2 < 0$ for $x \in \partial S_c$.
% \end{itemize}
% %
Then, for $h$ sufficiently small, all the critical points of $p_h$ in $\operatorname{int}(S_c)$ are
non-degenerate and there are no non-trivial critical points outside of $\operatorname{int}(S_c)$.
\end{proposition}
In order to simplify the discussion, from now on we focus on the
shifted random curves $X-X(0)$; however, with a little abuse of notation, we will
keep using the letter $X$ to mean $X-X(0)$. This choice is just made for
convenience as it significantly simplifies the proofs of many of the results
that we present. However, it is simple to extend any of the results from $H^1_0$ to $H^1$. 
Following this notational convention, $X$ thus belongs $P$-almost surely to the space $H^1_0([0,1])=\{x:[0,1] \to
\mathbb{R} \text{ such that } \|x'\|_{L^2} < \infty \text{ and } x(0)=0
\}$. The Poincar\'e inequality ensures that the semi-norm
$\|x'\|_{L^2}$ is in fact a norm on $H^1_0([0,1])$ 
and $\|x\|_{L^2} \leq C_p\|x'\|_{L^2}$ with $C_p=1$ (i.e. $H^1_0([0,1])$ can be
continuously embedded in $L^2([0,1])$). In the following, we denote
$\|x\|_{H^1_0}=\|x'\|_{L^2}$ for $x \in H^1_0$. Moreover, to alleviate the
notation, from now on we denote $L^2=L^2([0,1])$ and
$H^1_0=H^1_0([0,1])$. If the curves were not shifted so that $X(0)=0$, then they
would belong $P$-almost
surely to $H^1=H^1([0,1])=\{x:[0,1] \to
\mathbb{R} \text{ such that } \|x\|_{L^2}+\|x'\|_{L^2} < \infty \}$, which can
still be continuously embedded in $L^2$.

The main goal of this section is to show that the $L^2$ gradient flow associated to $p_h$ is
well-defined. In particular, we establish the following facts:
\begin{enumerate}
\item the $L^2$ gradient flow associated to $p_h$ is a flow in $H^1_0$
\item for any initial value in $H^1_0$, there exists exactly one trajectory of
  such flow which is a solution to the initial value problem of equation \eqref{eq:ivp}
\item for any initial value in $H^1_0$, the unique solution of the initial value
  problem of equation \eqref{eq:ivp} converges to a critical point of $p_h$ as
  $t \to \infty$ and
  the convergence is with respect to the $L^2$ norm
\item all the non-trivial critical points of $p_h$ are in $H^1_0$,
the support of $P$.
\end{enumerate}
These facts guarantee that the clusters described in equation
\eqref{eq:populationCluster} exist and are well-defined.

\begin{remark}
\label{remark:convergenceOfSolutions}
In general, in an infinite dimensional Hilbert space, the trajectory of the solution of an ordinary differential equation
such as the one of equation \eqref{eq:ivp} may not converge as $t \to
\infty$. In fact, such trajectory can be entirely contained in a closed and bounded set
without converging to any particular point of that set. To guarantee the
convergence of the gradient flow trajectories, one needs that (see \citealp{jost2011riemannian})
\begin{enumerate}
\item the trajectories satisfy some compactness property
\item the functional of interest (in our case $p_h$) is reasonably well-behaved:
  for instance it is smooth, with isolated critical points.
\end{enumerate}
In $L^2$, compactness is a delicate problem: no closed bounded ball in $L^2$
is compact.
However, any closed and bounded
$H^1$ ball is compact with respect to the $L^2$ norm (and so is any closed and
bounded $H^1_0$ ball). In fact, $H^1$ can be
compactly embedded in $L^2$ (see, for instance, Chapter 5.7 of \citealp{evans1998partial}), which means that
every bounded set in $H^1$ is totally bounded in $L^2$ and $H^1$ can be
continuously embedded in $L^2$. Since
$H^1_0$ is a closed subspace of $H^1$, $H^1_0$ can also be compactly embedded in
$L^2$. From a theoretical point of view $L^2$ is strictly larger than
$H^1$. 
%(in fact, $H^1$ curves are assumed to be differentiable). 
However,
 % from a practical standpoint, there is little loss when one focuses on $H^1$: 
%assuming some degree of smoothness is common practice in most nonparametric statistical models;
%furthermore, 
$H^1$ is dense in $L^2$.
\end{remark}

The remainder of our discussion focuses on the main results of this section,
which concern the computation of the derivatives of $p_h$ and their properties,
the existence, the uniqueness, and the convergence of the solution of the initial value problem of
equation \eqref{eq:ivp}. 

Before we state our results, let us recall that for a functional random variable $X \sim P$ valued in $L^2([0,1])$
the expected value of $X$ is defined as the element $E_P \,X \in
L^2([0,1])$ such that $E_P \,\langle X, y \rangle_{L^2} = \langle E_P \, X,
y \rangle_{L^2}$ for all $y \in L^2([0,1])$ (\citealp{horvath2012inference}). Furthermore,
the expectation commutes with bounded operators. Also, recall that for a
functional $F$ mapping a Banach space $B_1$
into another Banach space $B_2$, the Frech\'et derivative of $F$ at a
point $a \in B_1$ is defined, if it exists, as the bounded linear
operator $DF$ such that
$\|F(a+\delta)-F(a)-DF(\delta)\|_{B_2}=o(\|\delta\|_{B_1})$. The most common
case in this paper sets $B_1=L^2$, $B_2=\mathbb{R}_+$, and $F=p_h$.
Because $DF$ is a bounded linear operator, if $B_1$ is also an Hilbert
space then the Riesz representation
theorem guarantees the existence of an element $\nabla F(a) \in B_1$
such that, for any $b \in B_1$, $DF(b) = \langle b, \nabla F(a)
\rangle_{B_1}$. The element $\nabla F(a)$ corresponds to the gradient
of $F$ at $a \in B_1$. In this way, the gradient $\nabla F(a)$ and the first derivative operator
$DF$ at $a \in B_1$ can be
identified. In the
following, with a slight abuse of notation, we will use $DF$ both
to mean the functional gradient of the operator $F$ (which is an element of
$B_1$) and its Frech\'et
derivative (which is a bounded linear operator from $B_1$ to $B_2$). It will be clear from the context whether we are
referring to the derivative operator or to the functional gradient. Note that higher
order Frech\'et derivatives can be similarly identified with multilinear operators on $B_1$ (see, for example, \citealp{ambrosetti1995primer}).

Recall that, by assumption, the function $K_h(\|X-x\|^2_{L^2})$ is bounded
from above by a constant $K_0$. Furthermore, it is three times differentiable
and its first Frech\'et derivative at $x$ is
\begin{equation}
\label{eq:derivativeOfk}
D K_h(\|X-x\|^2_{L^2}) = 2 K_h'\left( \|X-x\|^2_{L^2} \right) (x-X).
\end{equation}
The second Frech\'et derivative at $x$ corresponds to the symmetric bilinear operator
\begin{equation}
\label{eq:secondDerivativeOfk}
\begin{aligned}
&D^2K_h(\|X-x\|^2_{L^2}) (z_1, z_2)  = 2 K_h'\left( \|X-x\|^2_{L^2}  \right)
\langle z_1,z_2 \rangle_{L^2} \\
&+4 K_h''\left( \|X-x\|^2_{L^2}
  \right) \langle  x-X, z_1 \rangle_{L^2}\langle x-X ,z_2 \rangle_{L^2}
\end{aligned}
\end{equation}
for $z_1 ,z_2 \in L^2$.
\begin{remark}
Any bounded bilinear operator $B$ on $L^2$ can be represented as a bounded linear
operator from $L^2$ to $L^2$. In fact, let  $z_1$ be any element of $L^2 $;
then, $B(z_1,\cdot)$ is a bounded linear operator from $L^2$ to $\mathbb R$. By
the Riesz representation theorem, one can define $B(z_1) \in L^2$ by letting 
$\langle B(z_1), z_2\rangle_{L^2} =B(z_1,z_2)$ for any $z_2 \in L^2$. The
operator norm of $B$ is then defined by  
\begin{equation}
\|B\|=\sup_{\{v \ : \|v\|_{L^2}=1\}} \|B(v)\|_{L^2}.
\end{equation}
\end{remark}

It is straightforward to check that both derivatives
correspond to bounded linear operators under assumption (H1). The following
Lemma provides the first and the second Frech\'et derivatives of $p_h$.
%
% \begin{equation}
% \label{eq:boundOnFirstDerivativeOfk}
% Dk(y;x) \leq \frac{2}{h} K_1 (\|X\|_{L^2} +
% \|x\|_{L^2})\|y\|_{L^2}
% \end{equation}
%
% and 
% %
% \begin{equation}
% \label{eq:boundOnSecondDerivativeOfk}
% D^2k(y,z;x) \leq
% \left(\frac{4}{h^2}K_2(\|X\|_{L^2}+\|x\|_{L^2}) + \frac{2}{h}K_1 \right)\|y\|_{L^2}\|z\|_{L^2},
% \end{equation}
% %
% hence both derivatives are $P$-almost surely
% bounded.
% hence the first derivative is  $P$-almost surely bounded.
%

% \subsection{Limit of the trajectories of the gradient flow}
% First of all, we establish the compactness of trajectories.
%
\begin{lemma}
\label{lemma:derivativesOfp} 
Under assumption (H1)
the Frech\'et derivative of $p_h: L^2 \to \mathbb R$ at $x$ corresponds to the
$L_2$ element
\begin{equation}
\label{eq:derivativeOfp}
Dp_h(x) = 2E_P \, K_h'\left( \|X-x\|^2_{L^2}\right)( x - X ).
\end{equation}
The second Frech\'et derivative of $p_h$ at
$x$ corresponds to the symmetric bilinear operator
\begin{equation}
\label{eq:secondDerivativeOfp}
\begin{aligned}
D^2p_h(x)(z_1,z_2)=E_P \, &\left[ 4K_h''\left( \|X-x\|^2_{L^2}\right)  \langle x-X,z_1
\rangle_{L^2}\langle x-X,z_2 \rangle_{L^2} \right.\\
&\left.+ 2K'_h \left( \|X-x\|^2_{L^2}\right)\langle z_1,z_2\rangle_{L^2} \right].
\end{aligned}
\end{equation}
Furthermore, both derivatives have bounded operator norm for any $x\in L^2([0,1])$.
\end{lemma}
We state without proof the following standard Lemma.
\begin{lemma}
\label{lemma:auxiliaryToL2GradientInH1}
Let $v \in C_c^\infty([0,1])$ be a compactly supported infinitely
differentiable function. Suppose $f \in L^2([0,1])$ is such that
$\langle f,v'\rangle_{L^2} = L(v)$ for any such $v$, where $L \in L^2([0,1])^*$ is a bounded linear operator.
% with $\langle g,v \rangle_{L^2} \leq C\|v\|_{L^2}$ for some $C>0$.
Then the weak first derivative $f'$ of $f$ exists, $f' \in
L^2([0,1])$, and $\|f'\|_{L^2}=\|L\|_{(L^2)^*}$. Moreover, $\langle
f',v\rangle_{L^2}=-L(v)$ 
for any $v \in C_c^\infty([0,1])$ and therefore for any $v
\in L^2([0,1])$.
\end{lemma}
The following Proposition shows that the $L^2$ gradient of $p_h$ is an
element of $H^1_0$. Intuitively, this means that if the starting point of the
initial value problem of equation \eqref{eq:ivp} is in $H^1_0$ (and a solution
exists for that starting point), then we should expect that the path $\pi_x$
only visits elements of $H^1_0$, i.e. the $L^2$ gradient flow associated to
$p_h$ is a $H^1_0$ flow.
\begin{proposition}
\label{proposition:gradientIsInH1}
For any $x \in H^1_0$, the $L^2$ gradient of $p_h$ at $x$, $D p_h(x)$, is an element of
$H^1_0$ such that for any $y \in L^2$, 
\begin{equation}
\begin{aligned}
\langle Dp_h(x)' ,y\rangle_{L^2} = E_P \left[-2K'_h(\|X-x\|^2_{L^2}) \langle x'-X',y\rangle_{L^2}\right].
\end{aligned}
\end{equation}

%  Furthermore, its weak first derivative $\nabla p_h(x)'$
% has the form
% %
% \begin{equation}
% \nabla p_h(x)' = E_P\,\frac{2}{h}K' \left( \frac{\|X-x\|^2_{L^2}}{h}\right)(x'-X').
% \end{equation}
\end{proposition}
%
% Proposition \ref{proposition:gradientIsInH1} guarantees that if a
% solution $\pi_x$ to the problem of equation \eqref{eq:ivp} exists,
% then $\pi_x(t) \in H^1_0([0,1])$ for all $t \geq 0$ as long as
% $x=\pi_x(0) \in H^1_0([0,1])$. This means 
Proposition \ref{proposition:gradientIsInH1} also implies that the equation 
$\frac{d}{dt}\pi_x(t) =Dp_h(\pi_x(t))$ is meaningful when
restricted to $H^1_0$. The next Lemma, Lemma \ref{lemma:gradientIsLocallyLipschitz}, establishes that $Dp_h$ is locally
Lipschitz under the $H^1_0$ norm. The subsequent Lemma, Lemma
\ref{lemma:flowPointsTowardsBall}, guarantees that a solution of the
problem (if it exists) is necessarily bounded. These two Lemmas allow us to claim
that if the starting point $\pi_x(0)=x$
is an element of $H^1_0$, then the initial value problem of equation
\eqref{eq:ivp} has a unique solution in $H^1_0$. This claim is summarized
in Proposition \ref{proposition:existenceuniqueness}.
\begin{lemma}
\label{lemma:gradientIsLocallyLipschitz}
Under (H1),the $L^2$ gradient of $p_h$ corresponds to a locally
Lipschitz map in $H^1_0([0,1])$.
\end{lemma}
%
% \begin{remark} Let $H$ be a Hilbert space. The standard ODE theory guarantees that if $\Phi : H \to H$ is locally Lipschitz, then the solution
% to the ode $x'(t)= \Phi(x(t))$ exists for any given initial conditions $x(0)=x_0 \in H$. So lemma \ref{lemma:gradientIsLocallyLipschitz} establishes
% the existence and uniqueness of 
% the gradient.
% \end{remark}
%
%
\begin{lemma}\
\label{lemma:flowPointsTowardsBall}
The following two results hold under (H1) and (H2)
\begin{itemize}
\item[1.]
Suppose that $p_h(\pi_x(0)) \geq\delta >0$. If $\|\pi_x(t)\|_{H^1_0} \geq K_2N_1/\delta$,
then\\
$\langle Dp_h(\pi_x(t)),\pi_x(t)\rangle_{H^1_0} \leq 0$. 
\item[2.] Let $M>0$.  If $\|X\|_{H^1_0}\leq M$ a.s., then $\langle Dp_h(\pi_x(t)), \pi_x(t) \rangle_{H^1_0} \leq 0 $
as soon as $\|\pi_x(t)\|_{H^1_0}>M$.
\end{itemize}
\end{lemma}
The intutive interpretation of Lemma \ref{lemma:flowPointsTowardsBall} is that
a trajectory $\pi_x$ that is a solution to the initial value problem of
equation $\eqref{eq:ivp}$ cannot wander too far from the origin in $H^1_0$. In
fact, if the $H^1_0$ norm of $\pi_x$ increases too much, then the path $\pi_x$ is eventually pushed back
into the closed and bounded $H^1_0$ ball of radius $K_2N_1/\delta$ (or radius $M$ if
one makes the stronger assumption that the probability law $P$ of the random curves is completely concentrated on
the $H^1_0$ ball of radius $M$). This ``push-back'' effect is 
captured by the condition $\langle Dp_h(\pi_x(t)),\pi_x(t)\rangle_{H^1_0} \leq
0$. By combining Lemma \ref{lemma:gradientIsLocallyLipschitz} and Lemma \ref{lemma:flowPointsTowardsBall}, we obtain the following 
\begin{proposition}
\label{proposition:existenceuniqueness}
 Under assumptions (H1), (H2), and (H3),
the initial value problem $\pi_x'(t) = Dp_h(\pi_x(t))$ with $x=\pi_x(0) \in
H^1_0$ has a unique solution in $H^1_0$ with respect to the $H^1_0$ topology. 
Moreover, if $\|x\|_{H^1_0} \leq R$, then $\|
\pi_x(t)\|_{H^1_0} \leq C_1$ for all $t \geq 0$, where $C_1=C_1(R,K_2 ,N_1,p_h(x))$.
\end{proposition}
\begin{remark} 
Proposition \ref{proposition:existenceuniqueness} establishes the existence and
the uniqueness of a solution to the initial value problem of equation
\eqref{eq:ivp} in the $H^1_0$ topology. The initial
value problem can be solved uniquely in the $L^2$ topology as well. In fact, it is easily
verified that, because $D^2 p_h$ is bounded, then the first derivative of $p_h$, $D p_h :L^2 \to L^2$, is uniformly Lipschitz with respect to the
$L^2$ norm.
Thus, one only has to show that the $H^1_0$ flow $\pi_x$ of Proposition \ref{proposition:existenceuniqueness} solved in the $H^1_0$ topology corresponds to
the $L^2$ gradient flow associated to $p_h$.
To verify this, one needs to check that
 the $H^1_0$ solution also satisfies the initial value problem of equation
 \eqref{eq:ivp} under the $L^2$ norm. Specifically, consider the $H^1_0$
 solution $\pi_x$ of Proposition \ref{proposition:existenceuniqueness} with any $\pi_x(0)=x \in H^1_0$. The path $\pi_x$ is continuously
 differentiable as a map from 
 $\mathbb R_+$ to $H^1_0$. It suffices to check that $\pi_x(t)$ is continuously differentiable 
 as a map from $\mathbb R_+$ to $L^2$ as well. This is easily established using Poincar\'e inequality
 since
 \begin{equation}
   \label{eq:H1solutionIsL2solution}
   \begin{aligned}
     &\|\pi_x(t+\delta) -\pi_x(t) +D p_h(\pi_x(t)) \|_{L^2}\\
     & \le  C_p\|\pi_x(t+\delta) -\pi_x(t) +D p_h(\pi_x(t)) \|_{H^1_0} =o(\delta),
   \end{aligned}
 \end{equation}
where the Poincar\'e constant is $C_p= 1$ for the pair $(L^2, H^1_0)$.
It is clear from equation \eqref{eq:H1solutionIsL2solution} and the definition of Frech\'et derivative that the $H^1_0$ solution
$\pi_x$ also satisfies the initial value problem of equation \eqref{eq:ivp}
under the $L^2$ norm. Thus, $\pi_x$ is the unique $L^2$ solution of
the intial value problem of equation $\eqref{eq:ivp}$.
\end{remark}

% Proposition \ref{proposition:existenceuniqueness} guarantees the existence and
% the uniqueness of a solution $\pi_x$ to the initial value problem of equation
% \eqref{eq:ivp} with starting point $x \in H^1_0$. 
The following Theorem, based
on Proposition \ref{proposition:existenceuniqueness}, guarantees
the convergence of $\pi_x$ to a critical point of $p_h$ as $t \to \infty$. The
statement about the convergence strongly relies on the compact embedding of $H^1_0$ in $L^2$, the
boundedness of the first two derivatives of $p_h$, and assumption (H4).
\begin{theorem}
\label{theorem:solutionStaysInClosedBall}
Assume  (H1), (H2), (H3), and (H4) hold. Let $\pi_x$ be the $H^1_0$  solution of the initial value problem of equation
\eqref{eq:ivp} with $x=\pi_x(0) \in H^1_0$. Let $C_1>0$ be such that 
$\| \pi_x(t) \|_{H^1_0} \leq C_1$ for all $t \geq 0$. Then there exists a unique
$\pi_x(\infty) \in L^2$ 
such that $\|\pi_x(\infty)\|_{H^1_0} \leq C_1$, $\lim_{t \to
  \infty}\|\pi_x(t) - \pi_x(\infty)\|_{L^2}=0$, and
$Dp_h(\pi_x(\infty))=0$.
\end{theorem}

The results above show that the $L^2$ gradient flow on $p_h$ is well-defined and its trajectories converge to 
critical points of $p_h$ that are in $H^1_0$ whenever the starting point
$x=\pi_x(0)$ is an element of $H^1_0$. We conclude this section with the following Lemma
which states that all the non-trivial critical points of $p_h$ belong to
$H^1_0$: thus, even though the functional $p_h$ ``spreads'' the probability law
$P$ of the random curves outside of its support $H^1_0$ (in fact, it is easily seen that there exists points $x \in L^2$
that are not in $H^1_0$ with $p_h(x)>0$), yet all of its non-trivial critical
points still lie in the support of $P$.
\begin{lemma}
\label{lemma:criticalPointsInH1}
Assume (H1), (H2), and (H3) hold. Let $x \in L^2$ be a critical point of $p_h$ such that
$p_h(x) >0$ (i.e. $x$ is a nontrivial critical point of $p_h$). Then $x\in
H^1_0$. Furthermore, if $\|X\|_{H^1_0} \leq M$ $P$-almost surely, then all the
nontrivial critical points of $p_h$ are contained in $B_{H^1_0}(0,M)$.
\end{lemma}
Note that the stronger assumption that $P(\|X\|_{H^1_0} \leq M)=1$ is a
functional analogue of the boundedness assumption which is frequently made
with finite-dimensional data. 

\section{Finite-dimensional adaptivity}
\label{section:weakAdaptivity}
%
%In this section we will assume that the distribution of random functions $X$ are supported in some finite dimensional vector space. In other words, 
%\begin{equation}
%P(X\in S)=1,
%\end{equation} where $S $ is a finite dimensional subspace of $L^2$.
In this section, we assume that the distribution of the random function $X$ is supported on some compact subset $S_c$ of a finite dimensional vector space. In other words, 
\begin{equation}
P(X\in S_c)=1,
\end{equation}
where $S_c$ is a compact subset of a finite-dimensional subspace $S \subset L^2$.
Two insightful outcomes are discussed in detail.
\begin{itemize}
\item[1.] Under some mild extra assumptions on the finite dimensional distribution of $X$, it is shown in Lemma \ref{lemma:nonDegeneracyOfCriticalPoints} that $p_h$, as a functional from $L^2$ to $\mathbb{R}_{+}$, is a Morse functional. This provides an important sufficient condition under which (H4) holds.
\item[2.] If the functional random variable $X$ admits a finite-dimensional distribution on $S_c$,  
it is natural to
ask whether the $L^2$ gradient flow on $p_h$ corresponds to the
finite-dimensional gradient flow associated to the expectation of a kernel density estimator of the density of $X$ on $S$. This
section provides a positive answers to this question. Furthermore, we show that such finite-dimensional gradient flow is entirely contained in $S$.
% Further changes to [2.]?
\end{itemize}
Suppose that the probability law $P$ of the functional random variable $X$ is
supported on a compact subset $S_c$ of a finite-dimensional space $S \subset L^2$.
If this is the case, there exists $\delta >0$ such that 
if $0< h \leq \delta$, then $p_h$ is a Morse function on the interior of $S_c$ (see Remark
\ref{remark:commentsOnMorseCondition} and Proposition
\ref{proposition:phIsMorse}). Moreover, as implied by
Lemma \ref{lemma:Sexample} and Lemma \ref{lemma:nonDegeneracyOfCriticalPoints} of this section, the
trajectories of the $L^2$ gradient flow associated to $p_h$ are all contained
in $S$ and they end at critical points of $p_h$ that belong to $S_c$. It is natural to
ask whether the $L^2$ gradient flow on $p_h$ corresponds to the
finite-dimensional gradient flow associated to some pseudo-density on $S$. This
section answers this question and shows that, if $X$ admits a density function $p$
(when $X$ is viewed as a finite-dimensional random vector in $S_c$), then the $L^2$ gradient flow associated
to $p_h$ corresponds to the gradient flow associated to the expectation of a kernel
density estimator of $p$ with bandwidth $h$.

Let $S = \text{span} \{ f_1\ldots f_d\}$ be a linear subspace of $L^2$. Without
loss of generality, assume that the $f_i$'s form an orthonormal 
basis of $S$ equipped with the $L^2$ norm and that $X\in S_c$
almost surely. Then, $X$ admits the decomposition $X=a_1 f_1 +\ldots +a_df_d$
for some random coefficients $\{a_i\}_{i=1}^d$. Let $\tilde X = [a_1,\dots,
a_d]^T$ and suppose that the distribution of $\tilde X$ has density $p : \mathbb
R^d \to \mathbb R_+$ with respect
to the Lebesgue measure. We have
 % the following Lemma.
%
\begin{lemma}
\label{lemma:Sexample}
Assume (H1) and (H2) hold and $P(X\in S_c) =1$. If $x=\pi_x(0) \in S$, then $\pi_x(t) \in S$ for any $t
\geq 0$. Furthermore, all the non-trivial critical points of $p_h$ belong to $S$.
\end{lemma}
%
% Following \cite{chacon2012clusters} and
% \cite{chacon2014populationbackground}, in this section we replace assumption
% (H4) with the following assumption:
% %
% \begin{itemize}
% \item[(H4')]
% $p$ is a Morse function on $S$.
% \end{itemize}
% %
For the rest of this section, let us replace assumption (H4) with
\begin{itemize}
\item[(H4')]
$X$ is an element of $S_c$ with probability 1, $X \sim P$ admits
density $p$ on $S_c$, and $p$ satisfies the assumptions of Proposition \ref{proposition:phIsMorse}.
\end{itemize}
%
% \begin{remark} 
% \label{computationinS}
Consider $x = x_1f_1+\dots +x_df_d \in S$. Let $\tilde x =[x_1,\dots, x_d]^{T}
\in \mathbb{R}^d$. Define $\tilde p_h(\tilde x): \mathbb R^d \to \mathbb R_+$ 
to be $\tilde p_h (\tilde x) = E_P \, K_h(\|\tilde X- \tilde x\|_2^2)$, where
$\|\cdot\|_2$ denotes the standard Euclidean norm. Note that
$\frac{1}{h^d}\tilde p_h(\tilde x)$ is the expectation of a standard finite
dimensional kernel density estimator at $\tilde x$. Since
$\|X-x\|^2_{L^2}=\|\tilde X -\tilde x\|^2_2$, it is clear that $p_h(x) =\tilde p_h(\tilde x)$.
To see the connection between the functional gradient $Dp_h(x)$ and $\nabla
\tilde p_h(\tilde x)$, the gradient of $\tilde p_h$ at $\tilde x$, note that the random variable
\begin{equation}
\begin{aligned}
\label{derivativeequivalent}
\langle Dp_h(x), f_i\rangle_{L^2} &=\langle  E_P \, 2K_h'( \|X-x\|^2_{L^2}) (x-X), f_i\rangle_{L^2} \\
&=  E_P \, 2K_h'( \|X-x\|^2_{L^2}) \langle x-X , f_i \rangle_{L^2} \\
&= 2E_P K' (\|\tilde X-\tilde x\|_2^2) (x_i-a_i)
 \end{aligned}
 \end{equation}
agrees with the $i$-th component of the gradient of $\tilde p_h$ at $\tilde
x$. This equivalence implies that the gradient flow (with starting points in the subspace $S$) on
$p_h$ and $\tilde p_h$ coincide (note that scaling the $\tilde p_h$ by $h^{-d}$
does affect the associated gradient flow). Furthermore, there exists a $\delta >0$
depending on $p$ such that $\tilde p_h(\tilde x)$ is a Morse
function for $0< h \leq \delta$ (see Remark \ref{remark:commentsOnMorseCondition}). Therefore, all the  non-trivial critical points of
$\tilde p_h$ are separated in $\mathbb R^d$. In light of Lemma
\ref{lemma:Sexample}, all the non-trivial critical points of $p_h$ are thus separated
in $S$ (and in $L^2$). 

Next, we have the following Lemma which
guarantees that if $p$ is a Morse density on $S_c$, then the non-trivial critical points of $p_h$ are
non-degenerate for $h$ sufficiently small and they all belong to $S_c$ (a critical point $x^*$ of $p_h$ is non-degenerate if $D^2p_h(x^*)$ is an isomorphism from $L^2$ to $L^2$).
\begin{lemma}
\label{lemma:nonDegeneracyOfCriticalPoints}
Under assumption (H1) (H2) and  (H4'), all the non-trivial critical points of $p_h$ lie in
$S_c$ and are non-degenerate for $h$ sufficiently small. Thus, for sufficiently small $h$, (H4) holds.
\end{lemma}
% %
% \begin{remark}
%  From the above discussion, for a non-trivial critical point $x^*$ there exists $\epsilon>0$ such that
%  $B_{L^2} (x^*, \epsilon)$ contains no other non-trivial critical points of
%  $p_h$. Suppose that $p_h(x^*)=\delta>0$ and $p_h$ is Lipschitz with Lipschitz
%  constant $L$. Pick $\epsilon' = \min \{\epsilon, \delta / L \}$. Then, 
%  $B_{L^2} (x^*, \epsilon')$ does not contain critical points (neither trivial
%  nor non-trivial) other than $x^*$. In fact, if $p_h$ is Lipschitz and $x^*$ is
%  a non-trivial critical point, we have
% %
% \begin{equation}
% p_h(x) \geq p_h(x^*) + L \|x-x^*\|_{L^2}=\delta + L \|x-x^*\|_{L^2} > 0
% \end{equation}
% %
% if $\|x-x^*\|_{L^2} < \delta / L$. Thus, assumption (H4) of Section
% \ref{section:populationBackground} is satisfied if $P$ admits a Morse density on
% a compact subset $S_c$ of the finite-dimensional subspace $S$ and
% $p_h$ is Lipschitz.
%  \end{remark}
% %

In the finite-dimensional case considered in this section, we can say more about the behavior of the $L^2$
gradient flow on $p_h$. In particular, we can characterize the solutions to the
initial value problem of equation \eqref{eq:ivp} also for the case in which the
starting point $x=\pi_x(0)$ does not belong to the support of $P$ (which is, in this case,
$S_c \subset L^2$). In fact, let $x$ be an element of $L^2$ which does not belong
to $S$.
% If the starting point of a gradient ascent path associated to $p_h$ corresponds,
% for instance, to an estimate $\hat X \in S'$ obtained by smoothing a noisy
% random function $Z = X+\epsilon$ (where $X \sim  P$ and $\epsilon$ corresponds
% to e.g. a white noise
% process) that is only partially observed over its domain
% $[0,1]$, then such starting point generally does not belongs to $S$. Thus, it is worth
% analyzing the behavior of the gradient ascent path
% corresponding to an initial point $\hat X=\pi_{\hat X} (0) \in L^2$ that may not
% belong to the finite-dimensional subspace $S$. $\hat X$ could be, for
% instance, an estimate of
% $X$ based on the $m$ pairs $\{s_i, Z(s_i)\}_{i=1}^m$ with $s_i \in [0,1]$ for
% $i=1,\dots,m$. Once we fix $\hat X$, 
The Gram-Schmidt orthogonalization process
guarantees that there exists $S' \supset S$ such that
$x=\pi_{x}(0) \in S'$ and $S'= \text{span} \{f_1,\dots f_d, f_{d+1}\}$, where $f_{d+1}$
is orthogonal to $\{f_i\}_{i=1}^d$ and $\|f_{d+1}\|_{L^2}=1$. The following Lemma guarantees
that the gradient ascent path originating from $x$ is entirely contained in
$S'$. Its proof is identical to that of Lemma \ref{lemma:Sexample}.
\begin{lemma}
Assume (H1) and (H2) hold and that $P(X\in S_c) =1$. Suppose $x=\pi_{x}(0)\in S'$. Then $\pi_{x}(t)\in S'$ for all $t
\geq 0$.
\end{lemma}
\begin{remark}
In the finite-dimensional setting of this section (in particular under
assumption (H4')), and for $h$ sufficiently small,
the basin of attraction of a saddle point of $p_h$ is negligible: in fact, from
the above disussion, it is clear that if the random function $X \sim P$ is
valued in a compact subset $S_c$ of a finite-dimensional
linear subspace $S$ of $L^2$ and $P$ has a proper Morse density $p$ on
$S_c$, then the basin of attraction of any saddle point of $p_h$ is neglibible for $h$
sufficiently small (since $p_h$ is Morse on $\operatorname{int}(S_c)$ for $h$
small enough). Stated more precisely, for $h$ sufficiently small, if $x_0^* \in \operatorname{int}(S_c)$
is a saddle point of $p_h$ then
$P(\{x \in S: \lim_{t \to \infty} \|\pi_x(t) - x_0^*\|_{L^2}=0\})=0$.
\end{remark}

\section{Statistical relevance of the estimated local modes}
\label{section:confidenceMarking}
%
% In this section, it is not assumed that $p_h$ is a Morse function.
% The goal of this section is to provide a test on the collection of the critical points of the the population function
% $p_h$ to classify the critical points according to whether it is a non-degenerate local maximum,  degenerate local 
% maximum, or a saddle. The test will serve two purposes.
The empirical counterpart of $p_h(x)= E_P \, K_h(\|X-x\|^2_{L^2})$ is the functional $\hat p_h(x) = \frac{1}{n} \sum_{i=1}^n K_h(\|X_i-x\|^2_{L^2})$, where $\{X_i\}_{i=1}^n$ are i.i.d sampling from the probability law $P$.
The critical points of $\hat p_h$ can be found, for example, by using a functional version of the mean-shift algorithm (see \citealp{ciollaromeanshift}).
In this section, we provide a statistical algorithm to detect whether a critical point of $\hat p_h$ 
corresponds to a local maximum of $p_h$. This algorithm provides two insights for functional mode clustering.
\begin{itemize}
\item [1.] For finite-dimensional clustering problems, if the underlying density $p$ is a Morse function,
 then the basin of attraction of a saddle point of $p$ has null probability content as
 it corresponds to a manifold of lower dimension. 
 In functional data clustering, however, the structure of the functional space is more complicated in the sense that there is no guarantee that the
 probability content of the basin of attraction of a saddle point of $p_h$ is
 negligible, even if $p_h$ is a Morse function. However, in analogy with the finite-dimensional case, clusters associated to
  non-degenerate local modes should generally be considered more informative as opposed to clusters
  associated with saddle points.
%   (the notion of non-degeneracy is made precise
%  in Definition \ref{def:nondegeneracy}).
  \item[2.] Several results in the previous section are derived under assumption (H4), which essentially states that the local critical points of $p_h$ are well-behaved. Without assuming (H4), the algorithm provides a simple way to classify well-behaved local critical modes of $p_h$ by analyzing $\hat p_h$. In this way, informative clusters can still be revealed in a less restrictive setting.
\end{itemize}
%
% \begin{enumerate}
% \item \textcolor{red}{CLARIFY} Without taking the assumption $p_h$ being a Morse function, 
% the results shown in the inference section fail to hold. However, if most of the gradient flow
% which starts at  the data points converges to a significant non-degenerate local maximum, then
%  clustering  according to the stochastic process is still making a good sense.  
 % \item 
Since the local modes of $\hat p_h$ that correspond to
non-degenerate local modes of $p_h$ provide the greatest insight about the population
clustering, we refer to these local modes as ``significant'' local modes.
In the following, we derive a procedure that allows us to 
discriminate the significant local modes from the non-significant ones.
% A statistical test to detect whether a critical point of $\hat p_h$ is a relevant
%  local mode helps understanding the clustering structure induced by $\hat p_h$
%  and hence
%  (in fact, in analogy to the finite-dimensional case, clusters associated to
%  local modes should generally be considered more relevant as opposed to clusters
%  associated to saddle points).
% \end{enumerate}

Before giving the definition non-degeneracy for a critical point of a functional
defined on an Hilbert space ($L^2$ in our case), it is convenient to adopt the convention that 
a linear operator from an Hilbert space to itself can be associated to a bilinear form on the Hilbert space and vice versa.
For example if $T : L^2 \to L^2$ is a linear operator, then it can be associated
to a bilinear form by letting $T(v, w) = \langle Tv ,w\rangle_{L^2}$.
\begin{definition} 
Let $T : L^2 \to L^2$ be a bounded linear operator. $T$ is said to be self-adjoint if $\langle T v, w\rangle  = \langle v, Tw\rangle$. 
$T$ is said to be positive (respecively negative) definite if $\langle Tv,v\rangle > 0$ (respectively $<0$) for all $v\neq 0$. Furthermore, $T$ is said to be an isomorphism if both $T$ and $T^{-1}$ are bounded.
\end{definition}
\begin{definition}
\label{def:nondegeneracy}
 Let $f: L^2 \to \mathbb R$ be twice continuously
  differentiable with bounded third derivative. Suppose $x^*$ is a critical point of $f$, i.e. $Df(x^*)
  =0$. Then, $x^*$ is said to be a non-degenerate local maximum (respectively
  minimum) if $D^2f(x^*)$ is a negative (respectively positive) definite isomorphism on $L^2$.
\end{definition}
It is a known fact that for any $x$, the second derivative of $f$, $D^2f (x)$,
is a self-adjoint linear operator. Furthermore, the following 
Lemma follows as a simple consequence of the fact that the second
derivative of $f$ at a non-degenerate local maximum is a self-adjoint negative-definite isomorphism.
\begin{lemma}
\label{lemma:equivalentnorms}
Suppose that $x^*$ is a non-degenerate local maximum of $f$. Then there exist $\delta >0$ such that 
\begin{equation}
\sup_{\|v\|_{L^2} =1}D^2f(x^*) (v,v) \leq -\delta.
\end{equation}
\end{lemma}
Let now $f_1, f_2 : L^2 \to \mathbb{R_+}$ be twice continuously differentiable with
bounded third derivative. Consider the following abstract setting for $f_1$ and $f_2$.
\begin{itemize}
\item[(C1)] The non-trivial critical points of $f_1$ and $f_2$ are all in
  $H^1_0$.
 \item[(C2)] For $i=1,2$,
if $x\in H^1_0$ then
 $Df_i(x) \in H^1$. Moreover,
\begin{equation}
   \label{eq:twoodes}
  \begin{cases}
   \pi'_i(t) = D f_i(\pi_i(t)) \\
    \pi_i(0) \in H^1_0,
  \end{cases}
\end{equation}
have $H^1_0$ solutions whose trajectories admit a convergent subsequence in
$L^2$.
\item[(C3)] For $\ell=0,1,2$,  let $\eta_\ell$ denote
\begin{equation}
\eta _\ell = \sup_{x\in B_{H^1_0}(0,M)}\| D^\ell f_1 (x) -D^\ell f_2(x)\|,
\end{equation}
where $\| \cdot \|$ stands for  the appropriate norms.
 Also, for 
$i=1,2$ and $k=0,1,2,3$, let 
\begin{equation}
\beta _k = \sup_{x\in L^2}\| D^k f_i (x)\| <\infty.
\end{equation}
\end{itemize}
\begin{remark}
Of course, the results that we obtain here are most useful for the particular
case where
\begin{equation}
 \label{eq:f1f2example}
\begin{aligned}
  & f_1(x) = p_h(x) = E_P \, K_h(\|X-x\|^2_{L^2}) \\
  & f_2(x) = \hat p_h(x) = \frac{1}{n} \sum_{i=1}^n K_h(\|X_i-x\|^2_{L^2})
\end{aligned}
 \end{equation}
and $X_1, \dots, X_n \sim P$ are i.i.d. random variables valued in $H^1_0$.
In this case, Lemma \ref{lemma:criticalPointsInH1} and Proposition 
\ref{proposition:existenceuniqueness} provide sufficient conditions for
(C1) and (C2), respectively. The boundedness for $\beta_k$ is ensured by (H1), and the probability bounds in Appendix B guarantee $\eta_l$ 
converges to 0 as the sample size $n$ increases.
\end{remark}
\begin{lemma} 
\label{lemma:maxareclose}
Suppose conditions (C1), (C2) and (C3) hold. Let $x^*_2$ be a non-degenerate local maximum of $f_2$ such that $\|x^*_2\|_{H^1_0} <M$. By Lemma
\ref{lemma:equivalentnorms}, there exists $\delta(x^*_2) >0$ such that
$\sup_{\|u\|_{L^2} =1} D^2f_2(x_2^*) (u,u) := -\delta(x^*_2)<0$. If $\eta_1 \leq
\delta^2 (x^*_2)/(8\beta_3)$ and  $\eta_2\leq \delta(x^*_2)/ 8$, there
exists $x_1^* \in B_{L^2} (x_2^*,\delta(x^*_2)/(2\beta_3))$, such that 
\begin{enumerate}
\item $x_1^*$ is a unique local maximum of $f_1$
 in $B_{L^2} (x_2^*,\delta(x^*_2)/(2\beta_3))$
 \item $\sup_{\|u\|_{L^2} =1} D^2f_1(x_1^*) (u,u) \le -3\delta(x^*_2)/8$ 
 \item $\|x^*_1 -x^*_2\|_{L^2} \le 8 \eta_1/\delta(x^*_2)$.
\end{enumerate}
%  where $C>0$ is an
% absolute constant.
\end{lemma}
%
% \begin{corollary}
% \label{corollary:testofmax}
% Under the assumptions of Lemma \ref{lemma:maxareclose},
% \begin{equation}
%  \sup_{\|u\|_{L^2}=1} D^2 f_1(x^*_1) ( u, u)  \le  \sup_{\|u\|_{L^2}=1} D^2
%  f_2(x^*_2) ( u, u) +\frac{8\beta_3}{\delta} \eta_1 +\eta_2.
% \end{equation}
% \end{corollary}
%
% \begin{remark}
Consider $f_1(x)=p_h$, $f_2(x) =\hat p_h(x)$ as in equation \eqref{eq:f1f2example}. 
For any $\alpha \in (0,1)$, we can derive a procedure based on Lemma
\ref{lemma:maxareclose} which allows us to classify
non-degenerate local modes of $\hat p_h$ as significant and construct an $L_2$ neighbor around
them with the property that the probability that each of such neighbors contains
a non-degenerate local mode of $p_h$ is at least $1-\alpha$ for $n$ large enough. The procedure is
summarized in Display \ref{table:confidenceMarking} and its statistical
guarantees are described in Proposition \ref{proposition:confidenceMarking}.\\
%
% We describe in Table \ref{algorithm:confidenceMarking} a statistical procedure that flags a non-degenerate
% local mode of $\hat p_h$ as $(1-\alpha)$-relevant and returns a confidence set for the
% location of the nearby non-degenerate local mode $x_1^*$ of $p_h$,
% with the property that the probability that $x_2^*$ is correctly flagged as
% $\alpha$-relevant is at least $1-\alpha$.\\
% at which the curvature of $\hat p_h$
% is sufficiently large in magnitude, we can construct an $L^2$-neighborhood
% around it with the property that with probability at least $1-\alpha$ such
% neighborhood contains a non-degenerate critical point $x_1^*$ of $f_1=p_h$. The
% procedure is described in Table \ref{algorithm:confidenceMarking}.\\
\text{}\\
\begin{minipage}{\textwidth}
\fbox{\parbox{\textwidth}{
\begin{center}
	\textbf{Learning non-degenerate local modes}\\
\end{center}

\begin{description}
\item[\qquad  Input:] data, $X_1, \dots, X_n$; kernel function, $K$; bandwidth,
  $h>0$; \\
significance level $\alpha \in (0,1)$. 
\item[\qquad  Output:] a set $\mathcal{\hat R}$ of significant local modes of $\hat p_h$. 
\end{description}

\begin{enumerate}
\item Compute $\hat p_h$ 
and determine the
  set of non trivial local max of $\hat p_h$, $\mathcal{\hat C}$ (here non-trivial means $\hat x^*\in \mathcal {\hat C} \Rightarrow \hat p_h(\hat x^*)>0$).
\item If $\hat x^* \in
  \hat{\mathcal{C}}$ is such that $\delta(\hat x^*):=-\sup_{\|u\|_{L^2}=1} D^2 \hat p_h(\hat x^*) ( u, u)
  \geq  \max\{ \sqrt{8 \beta_3 C_1(\alpha)}, 8 C_2(\alpha)\}$
where
\begin{equation*}
  C_1(\alpha)= \left( \frac{125 M K_1^2K_2}{2n} \right)^{\frac{1}{3}}+\left(
    \frac{25 K_1^2 \log (\alpha/2)}{4n} \right)^{\frac{1}{2}}
\end{equation*}
and
\begin{equation*}
  C_2(\alpha)=\left( \frac{125 M K_2^2K_3}{4n} \right)^{\frac{1}{3}}+\left(
    \frac{25 K_2^2 \log \left(\alpha/2\right)}{8n} \right)^{\frac{1}{2}}
\end{equation*}
then classify $\hat x^*$ as a significant local mode of $\hat p_h$. Here $\beta_3 = 12K_3$.
\end{enumerate}
}
}
\captionof{table}{}
\label{table:confidenceMarking}
\end{minipage}

\begin{proposition}
\label{proposition:confidenceMarking}
Consider $f_1(x)=p_h$, $f_2(x) =\hat p_h(x) $. Assume (H1) and (H2) hold and $P(\|X\|_{H^1_0} \leq M)=1$
for some known $M>0$. Let $\mathcal{\hat R}$ denote the set of points classified by the algorithm of Display
\ref{table:confidenceMarking}. Then, for large enough n, with probability $1-\alpha$ the following holds for all
$\hat x^* \in \hat {\mathcal R}$:
\begin{itemize} 
\item[1.] the random ball
$B_{L^2}(\hat x^*,\delta(\hat x^*)/(2\beta_3))$ contains a unique non-degenerate
local mode $x^*$ of $p_h$
\item[2.]  $\|x^*-\hat x^*\|_{L^2} \le 8C_1(\alpha)/\delta(\hat x^*)$.
\end{itemize}\end{proposition}
Let $\mathcal R$ denote the set of non-degenerate local modes of $p_h$. Consider
 the map $ \Phi : \mathcal{ \hat R }\to \mathcal R$ by letting 
  \begin{equation}
  \label{eq:mapbetweenmax}
  \Phi(\hat x^*) =B_{L^2} (\hat x^*,\delta(\hat x^*)/(2\beta_3))  \cap \mathcal R\cap B(\hat x^*,\log (n)\, C_1 (\alpha)/\delta(\hat x^*)).
  \end{equation}
 According to Proposition \ref{proposition:confidenceMarking}, with probability $1-\alpha$, for every $\hat x^* \in\mathcal{\hat R}$, there exists a 
 unique  $x^*\in \mathcal R $ contained in the right hand side of  equation \eqref{eq:mapbetweenmax}. In other words, with probability $1-\alpha$, $\Phi$ is a well-defined map.
  Under suitable assumptions on $p_h(x)$, more can be said.
\begin{proposition}
\label{prop:type2ofmarking}
Assume that (H1) and (H2) hold and that $P(\|X\|_{H^1_0} \leq M)=1$ for some known $M>0$ . Suppose further that $p_h$ has finitely many non-degenerate local modes. Let $\mathcal R$ denote the collection of non-trivial local maxima of $p_h$.
 Then, with probability converging to 1 as $n \to \infty$, every $x^*\in \mathcal \mathcal R$ has a unique preimage of $\Phi$ in $\mathcal {\hat R}$.
\end{proposition}
\begin{remark}
Under the assumptions of 
Proposition \ref{proposition:confidenceMarking} and  \ref{prop:type2ofmarking}, one can conclude that with probability converging to $1-\alpha$,
the map $\Phi: \mathcal{\hat R} \to \mathcal \mathcal{R}$ is bijective. In other words, the algorithm in Display \ref{table:confidenceMarking} is consistent.
% Further more as shown in proof of the step 3 of proposition  \ref{lemma:type2ofmarking}, 
%\begin{equation}
%\|\hat x^* - \Phi(\hat x^*) \| \le c^{-1} C_1(\alpha)\log(n)n^{1/6} 
%\end{equation}
% for every $\hat x^* \in \mathcal {\hat R}$. Here $C_1(\alpha)=O(h^{-1}n^{-1/3} )$. Recall that the pseudo density is defined as
% \begin{equation}
% p_h=E_PK_h(\|X-x\|^2_{L^2}).
% \end{equation}
%The true density estimator should be $\frac{1}{\psi(h)}p_h(x)$
% where $\psi(h)$ here is the volume correction based on the underling dimension of $X$. For example if $X\in \mathbb R^d$, $\psi(h)= h^{d}$. If $X$ in functional space, $\psi(h)$ takes the form of $\psi(h)\sim \exp(-h^{-a})$. Taking the volume correction into account,
%\begin{equation}
%\|\hat x^* - \Phi(\hat x^*) \|_{L^2} \le C\frac{\log(n)n^{-1/6} }{\psi(h)h}.
%\end{equation}
%As mentioned before, the volume correction parameter has no effects on  the modal clustering result. However, it does affect the rate of convergent of the results.   
%Also see section 7 for the reasoning that $h$ should be pick as a  positive real number independent of $n$.
\end{remark} 

 \section{From theory to applications} 
 \label{sec:theoryToApplications}
So far, all the results have been developed in an infinite-dimensional functional space. In this section, we connect the theory that we developed to practical applications and, in particular, we address the following challenges.
%In short, accomplish the following two tasks.
%
 \begin{itemize}
 \item [1.] Complete functional data can never be observed: a functional data point is always observed at a set of discrete observations. For example, let $\{X_i\}_{i=1}^n$ be an i.i.d sample from a distribution $P$ on $H^1_0$ and let $\{t_j\}^m_{j=1}$ be a set of equally spaced design points. In practice, only noisy measurements of the $X_i$'s  at $\{t_j\}_{j=1}^m$ are available. It is therefore important to design procedures that work with discretized curves.
 \item[2.] While the theory is developed in an infinite-dimensional functional space, in practice any functional clustering method relies on the use 
of only finitely many basis functions. However, a flexible algorithm for functional data clustering should be asymptotically consistent with the infinite-dimensional theory.
 \end{itemize}
One way to accomplish these two tasks at the same time is to apply a projection method. As shown later in this section, projections onto a linear space
introduce small $L^2$ perturbations to the functional data and to the pseudo-density. Nonetheless, the procedure that we describe is tolerant to small perturbations (see Corollary \ref{corollary:consistenttildep} for more details). \\

Before turning to the technical arguments, let us describe the following simple example which motivates the projection approach.

 \begin{example}
 \label{example:motivatingRealData}
  Consider the simple model $y= X(t) +\epsilon$, where $X \sim P$ is a random function and $\epsilon$ is a random variable independent of $X$.
  Instead of observing $n$ complete random function samples $\{X_i\}^n_{i=1}$, one only observes the discrete noisy 
  measurements $\{y_{ij}\}$, where $y_{ij} =X_i(t_{j}) +\epsilon_{ij}$. Here, $\{t_{j}\}_{j=1}^m$ is a set of equally spaced design points for the samples and the measurement errors $\epsilon_{ij} \sim N(0,\sigma)$ are independent of $\{X_i\}_{i=1}^{n}$.  \\ \\

%Consider the simple model $y= X(t) +\epsilon$, where both $X$ and $\epsilon$ are random variables.
%Instead of observing n complete random function sample $\{X_i\}^n_{i=1}$, only discrete noisy 
%measurement  $\{y_{ij}\}$ are available, where $y_{ij} =X_i(t_{j}) +\epsilon_{ij}$.
%Here $\{t_{j}\}_{j=1}^m$ is the set of equal spaced design points for each sample. The independent random noises $\epsilon_{ij}\sim N(0,\sigma)$  are independent of $\{X_i\}_{i=1}^{n}$.  \\ \\
 %
 In \cite{gasser1984estimating}, for example, if one assumes further that the random function $X$ is bounded in $H^2$, i.e.
 \begin{equation}
 \int_{0}^1 |X''(t)|dt \le M_2 \quad P\text{-almost surely},
 \end {equation}
 it is shown that there exists a kernel $W$ so that an approximation of $X$ can be constructed as 
 \begin{equation}
\tilde X(t) = \sum_{j=1}^m\frac{ y_{ij}}{b} \int_{t_{j-1}}^{t_j}W\left (\frac{t-u}{b}\right)du.
 \end{equation} 
 If $b$ is chosen to be of order $m^{-1/5}$, the above estimator also satisfies 
\begin{equation}
 E(\|X-\tilde X_i\|_{L^2}^2| X) \le C(M_2,W) m^{-4/5}
\end{equation}
and 
\begin{equation}
\label{eq:H1expectationClose}
 E(\|X'-\tilde X'_i\|_{L^2}^2| X) \le C(M_2,W) m^{-2/5},
\end{equation}
where $C(M_2,W)$ is a constant only depending on $M_2$ and $W$, and the expectation is taken with respect to $\epsilon$ only.
 \end{example}
As shown in the example, with noisy discrete measurements of the functional datum $X$, one can construct an approximation 
 \begin{equation}
% \tilde X \in \text{span} \left\{ \int_{t_{j-1}}^{t_j}W\left(\frac{t-u}{m^{-1/5}}\right)du \right\}^m_{j=1}.
  \tilde X \in \text{span} \left\{ \int_{t_{j-1}}^{t_j}W\left(\frac{t-u}{b}\right)du \right\}^m_{j=1},
 \end{equation}
 where $b$ is of order $O(m^{-1/5})$.
  This  approximation 
%  always introduces a small perturbation to
corresponds to a perturbed version of  
   the underlying complete functional datum. The perturbation vanishes asymptotically as the number of discrete measurements $m$ goes to infinity. This motivates the following assumption.
 \begin{itemize}
 \item[(H5)] The collection $\{\tilde X_i\}_{i=1}^n$ of the approximations of $\{X_i\}^n_{i=1}$ based on the equally spaced design points $\{t_j\}_{j=1}^m \subset [0,1]$ is such that $\{\tilde X_i\}_{i=1}^n \subset H^1_0$
% \subset B_{H^1}(0,M)$ 
 and
 \begin{equation} 
 E(\| X_i -\tilde X_i\|_{L^2}|X_i ) \le \phi (m),
\end{equation}
 where $\phi(m)$ does not depend on $\{X_i\}_{i=1}^n$ and $\phi(m) \to 0$ as $m \to \infty$.
 \end{itemize}
 %
%In order to keep the presentation simple, we assume that $\{\tilde X_i\}_{i=1}^n \subset B_{H^1}(0,M)$, which is slightly stronger than equation \eqref{eq:H1expectationClose}. However, it is not hard to replicate the following results by relaxing this assumption to a weaker one such as \eqref{eq:H1expectationClose}.

Recall that in our theoretical results, the sample version of the  pseudo density  takes the  form
 \begin{equation}
\hat p_h(x) = \frac{1}{n} \sum^{n}_{i=1} K_h (\| X_i-x\|^2_{L^2})
 \end{equation}
When the only available functional data are $\{\tilde X_i\}_{i=1}^n$ instead of $\{X_i\}_{i=1}^n$, one should consider   
 \begin{equation}
 \label{eq:tildeph}
 \tilde p_h(x) = \frac{1}{n} \sum^{n}_{i=1} K_h (\| \tilde X_i-x\|^2_{L^2}).
 \end{equation}

The following simple Lemma is useful to characterize the aforementioned $L^2$ perturbation and allows us to derive Corollary \ref{corollary:consistenttildep}.
 \begin{lemma}
 \label{lemma:connection} 
 Let $\{X_i\}_{i=1}^n$ be an i.i.d. sample of functional data.
 Under assumptions (H1), (H2), (H3), and (H5), for $l=0,1,2$, 
  \begin{equation}
  \begin{aligned}
P \left(\sup_{x \in L^2}\|D^l \hat p_h(x) -D^l \tilde p_h(x) \|\ge \epsilon \right) \le \frac{2^lK_{l+1}\phi(m)}{\epsilon}
  \end{aligned}
  \end{equation}
 where $\|\cdot \|$ stands for appropriate $L^2$ operator norms.
  \end{lemma}

 \begin{corollary}
 \label{corollary:consistenttildep}
 Consider a modified version of the algorithm of Display \ref{table:confidenceMarking}, where $\tilde p_h$ is  replaced by $\hat p_h$ and $C_1(\alpha), C_2(\alpha)$ are replaced by 
 \begin{equation}
 \begin{aligned}
 \tilde C_1(\alpha)&= 	\left( \frac{125 M K_1^2K_2}{2n} \right)^{\frac{1}{3}}+\left(
    \frac{25 K_1^2 \log (\alpha/4)}{4n} \right)^{\frac{1}{2}}+ \frac{8K_2\phi(m)}{\alpha}\\
 \tilde C_2(\alpha)&=\left( \frac{125 M K_2^2K_3}{4n} \right)^{\frac{1}{3}}+\left(
    \frac{25 K_2^2 \log \left(\alpha/4\right)}{8n} \right)^{\frac{1}{2}}+\frac{16K_3\phi(m)}{\alpha}.
    \end{aligned}
 \end{equation}
Let $\tilde {\mathcal R}$ be the significant local maxima learned by this modified version of the algorithm.
Then, the following statements are true.
\begin{itemize}
\item[1.]
The probability that all random balls
$B_{L^2}(\tilde x^*,\delta(\tilde x^*)/(2\beta_3))$ with $\tilde x^* \in \mathcal{\tilde R}$
%\cap B_{H^1}(0, M)$ 
contain a unique non-degenerate local mode $x^*$ of $p_h$ and that $\|x^*-\tilde x^*\|_{L^2} \le 8\tilde C_1(\alpha)/\delta(\tilde x^*)$ is at least $1-\alpha$ for sufficiently large $n$.

\item[2.]Consider the map $ \Phi : \mathcal{ \tilde R }
%\cap B_{H^1}(0,M)
\to \mathcal R$ such that
  \begin{equation}
  \label{eq:secondmapbetweenmax}
  \Phi(\tilde  x^*) =B_{L^2} (\tilde x^*,\delta(\tilde x^*)/(2\beta_3))  \cap \mathcal R\cap B(\tilde x^*,\log n\, \tilde C_1 (\alpha)/\delta(\tilde x^*)),
  \end{equation}
where $\mathcal R$ denotes the collection of non-degenerate local maxima of $p_h$.
Suppose further that $p_h$ has finitely many non-trivial local modes and that they are non-degenerate. Assume that $\{\tilde X_i\}_{i=1}^n \in B_{H^1_0}(0,M)$.  Then, with probability converging to 1 as $n \to \infty$,  every $x^*\in \mathcal R$ has a unique preimage  in $\mathcal {\tilde R}$ with respect to $\Phi$.
% \cap B_{H^1}(0,M)$.
\end{itemize}
\end{corollary}

\begin{remark}
Let $\tilde S=\text{span} \left\{ \int_{t_{j-1}}^{t_j}W\left(\frac{t-u}{b}\right)du \right\}^m_{j=1}$, with $b=O(m^{-1/5})$.
Although $\{\tilde X_i\}_{i=1}^n \subset \tilde S$, $\tilde p_h(x)$ defined in equation \eqref{eq:tildeph} is still a functional on $L^2$. It is desirable to have a method that does not use  infinitely many $L^2$ basis functions to compute a non-degenerate local maxima  $\tilde x^*$ of $\tilde p_h$ and
$\delta(\tilde x^*) :=- \sup_{\|u\|_{L^2}=1} D \tilde p_h(\tilde x^*)(u,u)$. 
%The discussion in Section \ref{section:weakAdaptivity} already provides the answers. 
 Lemma \ref{lemma:Sexample},  together with the assumption that $\{\tilde X_i\}_{i=1}^n \subset \tilde S$, ensures that all the non-trivial critical points of $\tilde p_h(x)$ belongs to $\tilde S$. Equation \eqref{eq:secondDerivativeOnOrtogonalComplement} of Appendix \ref{section:appendix} shows that for any $v\in (\tilde S)^{\bot}$,
\begin{equation}
D^2 \tilde p_h(x) (v,v) = \frac{2}{n}\sum_{i=1}^nK_h'(\|\tilde X_i-x\|^2_{L^2}) \|v\|^2_{L^2} < 0.
\end{equation}
Therefore, in analogy to the results of Section \ref{section:weakAdaptivity}, in order to classify the significant local modes of $\tilde p_h$ it is not 
required to consider infinitely many $L^2$ basis functions.
%(although the computational complexity grows exponentially with $m$).
\end{remark}

\begin{remark}
The assumptions of
%From the results outlined in 
Example \ref{example:motivatingRealData}
%, it does not immediately follow 
do not immediately guarantee that $\{\tilde X_i\}_{i=1}^n \subset B_{H^1_0}(0,M)$ as we assume for the second claim of Corollary \ref{corollary:consistenttildep}. However, simple computations show that
\begin{equation}
P\left( \{\tilde X_i\}_{i=1}^n \subset B_{H^1_0}(0,M)\right)\geq 1 - C n m^{-\frac{2}{5}}
\end{equation}
for some positive constant $C$. Therefore, as long as $n = o\left(m^{-\frac{2}{5}}\right)$, then the consistency result (the second claim) in Corollary \ref{corollary:consistenttildep} still holds.
\end{remark}

\section{On the choice of the pseudo-density}
\label{section:choiceOfPseudoDensity}
It is well-known that the Lebesgue measure does not exist in
infinite-dimensional spaces. As a consequence, a proper density function for a functional random
variable cannot generally be defined (\citealp{delaigle2010}). Developing a theory of modal clustering
for functional data necessarily requires a choice of a surrogate
notion of density that substitutes the probability density function associated with the
data. A pseudo-density should satisfy some basic differentiability
properties, so that one can studies the associated gradient flows. While we explicitly
choose to use the functional $p_h$ of equation \eqref{eq:ph}, one could in
principle choose to work with a different functional. The choice of the
pseudo-density is not an easy one, however. 

First of all, with particular emphasis on the setting
that we consider, we point out that, while tempting, one cannot naively assume that
the pseudo-density is $L^2$ differentiable (or even just continuous) and also 
vanishes as the $H^1_0$ norm diverges.
% Furthermore, we show that the pseudo-density
% implicilty defined by a frequently used factorization assumption on
% the small-ball probability function is
% generally not usable for modal clustering.
%
\begin{fact}
Let $p:L^2([0,1]) \to \mathbb{R}_+$ be a pseudo-density for the
functional random variable $X$ valued in $H^1_0([0,1])$ such that $p$ is
$L^2$ continuous and $p(x) \to 0$ as $\|x\|_{H^1_0} \to \infty$. Then
$p=0$ everywhere on $H^1_0([0,1])$.
\begin{proof}
Consider the sequence of functions $x_n(t)=n^{-1} \sin(n^2t)$ for $n \geq
1$ and $t \in [0,1]$. Clearly, $x_n \in H^1_0([0,1]$ for any $n \geq
1$. Notice that $\|x_n\|_{L^2} \to 0$ and $\|x_n\|_{H^1_0} \to \infty$ as $n \to \infty$. Thus, by
assumption, $p(x_n) \to p(0)=0$ as $n \to \infty$. Consider now
$z_n=y+x_n$ where $y \in H^1_0([0,1])$. We have $\|z_n\|_{L^2} \to
\|y\|_{L^2}$ as $n \to \infty$, hence $p(z_n) \to p(y)$ as $n \to
\infty$. However $\|z_n\|_{H^1_0} \geq \|x_n\|_{H^1_0}-\|y\|_{H^1_0}$ .
 Hence $\|z_n\|_{H^1_0} \to \infty$ as $n \to
\infty$. We thus have $p(z_n) \to p(y)=0$ as $n \to
\infty$ for any $y \in H^1_0([0,1])$, implying that $p$ is null on $H^1_0([0,1])$.
\end{proof}
\end{fact}
The argument above shows that requiring a pseudo-density to be $L^2$ continuous and to vanish outside of
$H^1_0$ necessarily leads to an uninteresting scenario for modal
clustering, despite the fact that these two requirements apparently sound reasonable at first
and carry some resemblance with the standard 
assumptions that are made on density functions in finite-dimensional problems.

Secondly, analyzing the asymptotic regime where $h=h_n\to 0$ makes little sense even in the most well-understoon situations. In fact,
let us consider the following two settings in which one typically chooses $h=h_n\to 0$ as $n\to \infty$.
 \begin{itemize}
\item[1.] If the law $P$ of $X$ is supported on a finite-dimensional space and admits a density $p$, then the bias of $\hat p_h$ is easy to compute and one can choose $h_n\to 0$ to balance the bias-variance trade-off. However, since $p$ is defined over a finite-dimensional vector space $S$, the gradient flow is not well-defined outside of $S$. 

As discussed in Section \ref{sec:theoryToApplications}, all of the observed functional data are reconstructed from noisy discrete measurements. 
As a result, neither the observed discrete measurements nor the reconstructed functional data are in $S$, and the gradient flow with respect to $p$ is not well-defined starting at any of these points.

\item[2.] If the law of $P$ of $X$ is supported on an infinite-dimensional space (for example $X$ is a diffusion process), then some authors (see, for instance, \citealp{gasser1998nonparametric} and
\citealp{ferratysurrogatedensity}) suggest to implicitly define a pseudo-density by assuming a particular factorization of the
small-probability function associated to $P$.
In particular, it is assumed that 
\begin{equation}
\label{eq:smallballprobabilityfactorization}
P(\|X-x\| \leq h) = p(x)\phi(h) + o(\phi(h))
\end{equation}
as $h \to 0$ for some pseudo-density functional $p$ depending only on the center of the ball $x$
and some function $\phi$ depending only on the radius of the ball
$h$.

In this second case, $p$ is
non-zero only on its domain $S_2$, which is typically taken to be a compact subset of
the infinite-dimensional functional space. Compact subsets of $L^2$ are singular in the sense that any $L^2$ closed ball is not compact. 
As a result, the pseudo-density $p$ is also singular, hence not $L^2$ differentiable.
One might then assume that $p$  is differentiable with respect to norm induced by $S_2$ and  study the gradient flow associated to $p$ using the $S_2$ topology. In light of the first point just given, 
%we only have 
one should then assume that
%to consider the case where 
$S_2$ is the closure of an open set of an infinite-dimensional functional subspace.
This leads to an even more serious problem: closed bounded balls in $S_2$ are not compact under the $S_2$ topology. The lack of compactness implies that the gradient ascent paths are not guaranteed to converge with respect to the $S_2$ topology.
 \end{itemize}
 %
%Therefore, $h>0$ is suitable because choosing $h_n \to 0$ leads to serious problems.
%Thus, in both cases, although $p_h$ converges to $p$ as $h=h_n\to 0$ and one may be tempted to analyze the population modal clustering (and, in particular, the gradient ascent flows) on $p$ instead of $p_h$, this leads to serious complications. 

% Few stochastic processes
% have trajectories for which equation
% \eqref{eq:smallballprobabilityfactorization} can be shown to hold (see for instance
% \citealp{estimating2006ferraty}, \citealp{ferratysurrogatedensity}, \citealp{bongiorno2015clustering}), thus
% such factorization of the small-ball probability function should be
% regarded as a moderately strong assumption. 

%
% We now demonstrate that using the
% pseudo-density $p$ of equation
% \eqref{eq:smallballprobabilityfactorization} generally does not allow
% to develop a theory of modal clustering for functional data.

The pseudo-density functional
$p_h(x)=E_P\,K\left(\frac{\|X-x\|^2_{L^2}}{h}\right)$ with $h>0$ is therefore the most natural
candidate to develop a theory of modal clustering of smooth random
curves in a density-free setting.
%In fact, $p_h$ has a concrete interpretation in terms
%of the expectation of a kernel density estimator and only requires to impose
%mild assumptions on the kernel function and on the moments of $X$. 
Furthermore, the functional $p_h$
corresponds to the functional discussed by \cite{hall2002depicting}, who
proposed a mode-finding algorithm for functional data and had the intuition
that their algorithm was approximating a gradient flow on the estimator $\hat
p_h$. 

Of course, from a practical point of view, one has to choose a value for $h$.
%the challenges of modal
%clustering for functional data mainly rely on the choice of the bandwidth
%$h$. 
It is well-known that, in finite-dimensional scenarios, the
behavior of the topological structure of $\hat p$ (the estimator of the
underlying density function) exhibits a phase-transition as $h$ varies: for
small values of $h$, the estimated density generates many irrelevant clusters while for
large values of $h$, $\hat p$ only generates few uninformative clusters which
eventually merge into a single one for $h$ large enough. Interestingly, one can usually empirically
identify a relatively broad range of intermediate
values of $h$ for which the number of clusters associated to $\hat p$ is
stable (see for instance, \citealp{genovese2013nonparametric}). Several criteria have been proposed in the
finite-dimensional literature to choose an optimal value of $h$ in
practice. These are usually based on two ideas: 1. using a cross-validation
criterion optimized for the estimation of $p$ and its derivatives
(\citealp{chacon2013data}) or 2. maximizing the number of statistically relevant modes of $\hat p$ as soon as
a statistical test for the relevance of the modes is available (\citealp{genovese2013nonparametric}). For functional data
clustering, we think that the latter approach is generally more appropriate and the algorithm that we propose in Section \ref{section:confidenceMarking} 
can be used this way to practically choose a value for $h$.

%(although somewhat heuristic), mostly due to 
%the current lack of theoretical guarantees and guidelines on cross-validation criteria for
%pseudo-density estimation. From a practical perspective, functional
%data problems also require to choose a norm or some other measure of
%distance. While this can be viewed as an element of flexibility in the analysis
%(see for instance \citealp{ferraty2006nonparametric} for a related discussion), it is also an
%element of complication: in practice, the distance measure is another parameter that one
%has to choose together with the smoothing parameter $h$.

\section{Discussion and conclusions}
\label{section:conclusionAndDiscussion}
In this paper, we provide a general theoretical background for clustering of
functional data based on pseudo-densities. We show that clusters of functional
data can be characterized in terms of the basins of attraction of the critical
points of a pseudo-density functional, both at the population and at the sample level. 
Our theory
can be generalized to 
%different choices of pseudo-density functionals and 
different
functional spaces, as long as the chosen pseudo-density functional is
sufficiently smooth and the range of the functional random variable $X$ can
be compactly embedded in a larger space to guarantee the compactness of the
gradient flow trajectories. 
Because of the need of a compact embedding, one has to consider
two non-equivalent topologies at the same time (in our case the $L^2$ and
the $H^1$ topologies): from a statistical
viewpoint, this means that the data need to be at least one order smoother than the space in which they are embedded.

Besides compactness, there is another element that makes the theory of population clustering in the
functional data setting more challenging when compared to the
finite-dimensional case. This is the fact that the basin of attraction of a
saddle point of $p_h$ is not necessarily negligible. While in the
finite-dimensional setting the basin of attraction of a saddle point of the
Morse density function $p$ is a manifold whose dimension is strictly smaller
than the dimension of the domain of $p$ (and therefore its probability content
is null), the same property is not necessarily satisfied by a pseudo-density
functional in the infinite-dimensional and density-free setting that we
consider. In analogy to the finite-dimensional case, one
would expect that clusters that are associated to the local modes of $p_h$ are
more relevant than those associated to the saddle points of the same
functional. It becomes natural to ask whether it is possible to derive a
statistical procedure that marks a local mode of $\hat
p_h$ (and its associated empirical cluster) as significant whenever it corresponds to a non-degenerate local mode of $p_h$. We provide a consistent
algorithm to achieve this task that can be applied to real data, such as noisy 
measurements of random curves on a grid.
Furthermore, although the sample pseudo-density is a functional with infinite-dimensional domain, 
the algorithm only requires to project the data onto 
a linear space that has finite dimension
%finitely many basis functions 
for any sample of curves of finite size $n$.
%The algorithm marks a set of non-degenerate
%local modes of $\hat p_h$ as relevant if the magnitude of the curvature at those
%local modes is sufficiently large. In addition, the algorithm outputs a
%confidence region around each local mode of $\hat p_h$ for the location of the
%associated non-degenerate local mode of $p_h$.

The analysis and the results of our
paper raise new relevant questions that will be addressed in future work.
%We argued that
%$E_P\,K\left(\frac{\|X-x\|^2_{L^2}}{h}\right)$ is a natural candidate, but it
%would be interesting to analyze the use of other
%functionals.
In particular, future work shall provide a systematic validation of our theory by means of 
applications on both simulated and real data, and an extension to multivariate functional data 
(i.e. curves in high-dimension). Such extension will further broaden the applicability of our results. 

%\cite{chacon2014populationbackground} proves a clustering
%consistency theorem in the one-dimensional setting. Roughly speaking, the theorem guarantees that,
%under suitable distances defined for partitions, the population clustering
%$\mathcal{C}$ can be consistently estimated by the empirical clustering
%$\hat{\mathcal{C}}$ if the population density $p$ and its first derivatives can
%be consistently estimated. It would be of great interest to extend this result both to
%the multivariate and to the functional setting. Finally, extending the results of \cite{ariascastromeanshiftgradient} about
%the uniformly consistent estimation of the gradient ascent paths $\pi_x$ to
%the functional case 
%represents a future theoretical challenge. 

\section*{Acknowledgements}
M. Ciollaro and D. Wang wish to thank Prof. Giovanni Leoni (Department of Mathematical Sciences, Carnegie Mellon University) and Xin Yang Lu (Department of Mathematics and Statistics, McGill University) for the very useful conversations
leading to several improvements in the early versions of this paper.

\newpage 

\appendix

\section{Proof of the results}
\label{section:appendix}

\begin{proof}[Proof of Proposition \ref{proposition:phIsMorse}]
For convenience, we prove the result under the additional assumption that $K$ is
compactly supported on $[0,1]$. The proof that we present can be easily extended to
exponentially decaying kernel functions and
this extra assumption can be safely removed.

Consider the 
% first 
set of assumptions of the Proposition. Furthermore, let $\bar K_2=\sup_{x \in \mathbb{R}^d}
\|\nabla^2 p(x)\|_{2}$ and $\underline K_1=\inf_{x \in \partial S_c} \|\nabla p(x)\|_{2}$. Note that since
$\partial S_c$ is compact, $\underline K_1>0$. Consider now the set $S_\epsilon=\{ x \in S_c
: d(x, \partial S_c) \geq \epsilon \}$. Then, there exists a set $\Omega$ such
that $S_{2\epsilon} \subset \Omega \subset S_\epsilon$ and $\partial \Omega$ is
also smooth. As a result, if $x \in \partial \Omega$, then $\epsilon \leq
d(x, \partial S) \leq 2\epsilon$ and $\inf_{x \in S \cap \Omega^c} \|\nabla p(x)\|_2
\geq \underline K_1 / 3$. Since $p$ is Morse on $\Omega$ and twice continuously
differentiable on $\operatorname{int}(S_c)$, then standard mollification results guarantee
that there exist $\eta>0$ and $h_1>0$ such that if $0< h \leq h_1$, then $\sup_{x \in
  \Omega} \| \nabla^{(i)} p_h(x) - \nabla^{(i)} p(x)\|_2 \leq \eta$ for $i=0,1,2$. Then, Lemma 16 of
\cite{chazal2014robust} guarantees that $p_h$ is Morse on $\Omega$. It is only
left to show that if $x \in \Omega^c$ is such that $\mathcal{L}(
B_{\mathbb{R}^d}(x,h) \cap S_c) > 0$ then $\nabla p_h \neq 0$. Consider $h
< \frac{\underline K_1}{6 \bar K_2}$ and let $n(\cdot)$ denote the outward
normal vector to $S_c$ with unitary norm. We have
\begin{equation}
\begin{aligned}
\nabla p_h(x) &= \nabla_x \int_{\mathbb{R}^d} K \left( \frac{\|y-x\|^2_{2}}{h}
\right)p(y)\,dy \\
&= \int_{S_c \cap B_{\mathbb{R}^d}(x,h)} \nabla_x K \left( \frac{\|y-x\|^2_{2}}{h}\right)p(y)\,dy \\
&= \int_{S_c \cap B_{\mathbb{R}^d}(x,h)} - \nabla_y K \left(
  \frac{\|y-x\|^2_{2}}{h}\right)p(y)\,dy \\
&=\int_{S_c \cap B_{\mathbb{R}^d}(x,h)} K \left( \frac{\|y-x\|^2_{2}}{h}\right)\nabla_y
p(y)\,dy \\
&-\int_{\partial S_c \cap B_{\mathbb{R}^d}(x,h)} K \left(
  \frac{\|y-x\|^2_{2}}{h}\right)n(y)p(y)\,dy \\
&- \int_{S_c \cap \partial B_{\mathbb{R}^d}(x,h)} K \left( \frac{\|y-x\|^2_{2}}{h}\right)
n(y)p(y)\,dy.
\end{aligned}
\end{equation}
Note that $p(y)=0$ if $y \in \partial S_c$ and, since $K$ is compactly supported
on $[0,1]$, $K \left( \frac{\|y-x\|^2_{L^2}}{h}\right)=0$ if $y \in
\partial B_{\mathbb{R}^d}(x,h)$. Hence, the last two integrals on the boundaries are null. Now,
since $\nabla p$ is $\bar K_2$-Lipschitz, we have
\begin{equation}
\begin{aligned}
\|\nabla p_h(x) \|_2 &= \left \| \int_{S_c \cap B_{\mathbb{R}^d}(x,h)} K \left( \frac{\|y-x\|^2_{2}}{h}\right)\nabla
p(y)\,dy \right \|_2 \\
& \geq -\left \| \int_{S_c \cap B_{\mathbb{R}^d}(x,h)} K \left( \frac{\|y-x\|^2_{2}}{h}\right)\nabla
p(y)\,dy \right. \\
&-\left. \int_{S_c \cap B_{\mathbb{R}^d}(x,h)} K \left( \frac{\|y-x\|^2_{2}}{h}\right)\nabla
p(x)\,dy \right \|_2 \\
&+\left\| \int_{S_c \cap B_{\mathbb{R}^d}(x,h)} K \left( \frac{\|y-x\|^2_{2}}{h}\right)\nabla
p(x)\,dy \right\|_2 \\
& \geq -\bar K_2 h \int_{S_c \cap B_{\mathbb{R}^d}(x,h)} K \left( \frac{\|y-x\|^2_{2}}{h}\right)
\, dy \\
& + \|\nabla
p(x)\|_2 \int_{S_c \cap B_{\mathbb{R}^d}(x,h)} K \left( \frac{\|y-x\|^2_{2}}{h}\right)\,dy \\
& \geq \left( \frac{\underline K_1}{3}-\bar K_2h\right) \int_{S_c \cap B_{\mathbb{R}^d}(x,h)} K \left( \frac{\|y-x\|^2_{2}}{h}\right)
\,dy > 0
\end{aligned}
\end{equation}
since $-\bar K_2 h > - \underline K_1/6$ and $\mathcal L(B_{\mathbb R^d} (x,h) \cap S_c) >0$. This shows that $p_h$ has no
non-trivial critical points outside of $\Omega$.

\end{proof}

\begin{proof}[Proof of Lemma \ref{lemma:derivativesOfp}]
%The first order Taylor expansion of $K_h$ around $x$ gives
%
\begin{equation}
\label{eq:TaylorExpansionOfk}
\begin{aligned}
&|K_h(\|X-(x+\delta)\|^2_{L^2}) | -K_h(\|X-x\|^2_{L^2}) \\
&- \langle D K_h(\|X-x\|^2_{L^2}),\delta \rangle_{L^2}| \\
&\leq \sup_{s\in [0,1]}\frac{1}{2}|D^2K_h(\|X-(x+s\delta)\|^2_{L^2}) (\delta,
\delta ) | 
\end{aligned}
\end{equation}
Now, using the bounds on the derivatives of $K_h$ and equation \eqref{eq:secondDerivativeOfk}, we have
\begin{equation}
\begin{aligned}
&|D^2K_h(\|X-(x+s\delta)\|^2_{L^2}) (\delta, \delta ) | \\
&\leq 4\left| K_h''(\|X-(x+s\delta) \|^2_{L^2})\right| \|X-(x+s\delta)\|_{L^2}^2 \|\delta\|^2_{L^2}\\
&+2\left| K_h'(\|X-(x+s\delta)\|^2_{L^2})\right| \|\delta\|^2_{L^2}\\
% &\leq C( ||x||_{L^2} ,E_P||X||^2_{L^2}, K_1, K_2) ||\delta||^2_{L^2}  =
% o(\|\delta\|_{L^2}).
&\leq 4 K_2 \|\delta\|^2_{L^2}  = o(\|\delta\|_{L^2}).
\end{aligned}
\end{equation}
Taking the expectation and applying Jensen's inequality in equation
\eqref{eq:TaylorExpansionOfk} yields
\begin{equation}
\begin{aligned}
&|E_P\, K_h(\|X-(x+\delta)\|^2_{L^2}) | - E_P \, K_h(\|X-x\|^2_{L^2}) \\
& - E_P\, \langle DK_h(\|X-x\|^2_{L^2}),\delta \rangle_{L^2}|= o(\|\delta\|_{L^2})
\end{aligned}
\end{equation}
which implies that
\begin{equation}
\begin{aligned}
\langle Dp_h(x),  \cdot \, \rangle_{L^2} &= E_P \langle D K_h(\|X-x\|^2_{L^2},
\cdot \ \rangle_{L^2} \\
& = E_P \langle 2 K_h'\left( \|X-x\|^2_{L^2} \right) (x-X) , \cdot \,
\rangle_{L^2} \\
& = \langle E_P  \, 2 K_h'\left( \|X-x\|^2_{L^2} \right) (x-X) , \cdot \,
\rangle_{L^2}.
\end{aligned}
\end{equation}
Thus, by definition, equation \eqref{eq:derivativeOfp} is established. It is
clear from assumption (H1) that $\|D p_h(x)\|_{L^2} \le 2K_1$.
In order to derive $D^2 p_h(x)$, a similar computation is used. The Taylor expansion
of $F(x) = K_h'( \|X-x\|^2_{L^2} )(x-X)$ as a function of $x$ gives
\begin{equation}
\begin{aligned}
\left \| F(x+\delta )- F(x)- D F(x) (\delta) \right\|_{L^2}  
\leq \sup_{s\in [0,1]}\frac{1}{2} \| D^2 F(x +s\delta ) (\delta, \delta)\|_{L^2}
\end{aligned}
\end{equation}
where
\begin{equation}
\begin{aligned}
 D F(x) (\delta) &= 2K_h''\left( \|X-x\|^2_{L_2}
  \right) \langle  x-X,\delta \rangle_{L_2} (x-X)  \\
  & +K_h'\left( \|X-x\|^2_{L_2}  \right) \delta.
  \end{aligned}
\end{equation}
Furthermore,
\begin{equation}  
\begin{aligned}
&D^2 F(x +s\delta ) (\delta_1, \delta_2) \\
&= 4K_h'''\left( \|X-x\|^2_{L_2}
  \right) \langle  x-X,\delta_1 \rangle_{L_2}\langle  x-X,\delta_2 \rangle_{L_2} (x-X)  \\
  &+2K_h''\left( \|X-x\|^2_{L_2}
  \right) \langle  \delta_1 ,\delta_2 \rangle_{L_2} (x-X) \\ 
   &+ 2K_h''\left( \|X-x\|^2_{L_2}
  \right) \langle  x-X,\delta_1 \rangle_{L_2} \delta_2  \\
  &+2 K_h''\left( \|X-x\|^2_{L_2}  \right)\langle x-X ,\delta_2 \rangle_{L_2} \delta_1
  \end{aligned}
\end{equation}
By assumption (H1), $\sup_{s\in [0,1]} \| D^2 F(x +s\delta ) (\delta, \delta)\|_{L^2} \leq 6K_3 \|\delta \|^2_{L^2}$ .
Thus, 
\begin{equation}
\begin{aligned}
&\|  E_P F(x+\delta )-  E_P F(x)-  E_P \, D F(x) (\delta)  \|_{L^2}  \\
&\leq E_P\, \|  F(x+\delta )-   F(x)-  D F(x) (\delta) \|_{L^2}  \leq  3K_3
\|\delta\|^2_{L^2},
\end{aligned}
\end{equation}
and the claim then easily follows.
\end{proof}
\begin{proof}[Proof of Proposition \ref{proposition:gradientIsInH1}]
\begin{equation}
\begin{aligned}
& \langle D p_h(x), v' \rangle_{L^2}=\left \langle E_P \, 2K'_h\left( \|X-x\|^2_{L^2}\right)
  (x-X), v' \right \rangle_{L^2} \\
& = E_P\,2K_h'\left( \|X-x\|^2_{L^2}\right)
  \langle x-X,v' \rangle_{L^2} \\
& = E_P \, -2K_h'\left( \|X-x\|^2_{L^2}\right) \langle x'-X',v
\rangle_{L^2} \\
& \leq E_P \, -2K_h'\left( \|X-x\|^2_{L^2}\right)
\|x'-X'\|_{L^2} \|v\|_{L^2} \\
& \leq 2 \, K_2(\|x'\|_{L^2}+E_P\|X'\|_{L^2})\|v\|_{L^2}  \\
& \leq 2 K_2(\|x'\|_{L^2}+N_1)\|v\|_{L^2} 
\end{aligned}
\end{equation}
where the second equality holds by integration by parts. An application
of Lemma \ref{lemma:auxiliaryToL2GradientInH1} with $L(v)=E_P \, -2K'\left( \|X-x\|^2_{L^2}\right) \langle x'-X',v
\rangle_{L^2}$ yields $\|D
p_h(x)\|_{H^1_0}=\|D p_h(x)'\|_{L^2} \leq
2K_1(\|x'\|_{L^2}+N_1)$ and therefore $D
p_h(x) \in H^1_0$.
\end{proof}
\begin{proof}[Proof of Lemma \ref{lemma:gradientIsLocallyLipschitz}]
Let $\|x\|_{H^1_0} , \|y\|_{H^1_0} \leq L < \infty$. It suffices to show
that $\exists 0<C(L)<\infty$ such that $\|D p_h(x) -D p_h(y)\|_{H^1_0} \leq
C(L)\| x-y \|_{H^1_0}$. Equivalently, one has to show
that $\|D p_h(x)' - D p_h(y)'\|_{L^2} \leq
C(L)\| x'-y' \|_{L^2}$. By Lemma
\ref{lemma:auxiliaryToL2GradientInH1} and
Proposition \ref{proposition:gradientIsInH1} we have that, for any $v \in L^2$,
\begin{equation}
\label{eq:main}
\begin{aligned}
& \langle D p_h(x)'-D p_h(y)',v \rangle_{L^2}  \\
&= 2E_P \, K_h'\left( \|X-x\|^2_{L^2}\right)
\langle x'-X',v \rangle_{L^2} \\ 
&- K_h'\left(
  \|X-y\|^2_{L^2}\right) \langle y'-X',v \rangle_{L^2} \\
&= 2  E_P \, \left[ K_h'\left( \|X-x\|^2_{L^2} \right)
      - K_h'\left( \|X-y\|^2_{L^2}\right) \right] \langle
    x'-X',v \rangle_{L^2}\\
 &+2E_P\,  K_h'\left( \|X-y\|^2_{L^2}\right)
\langle x'-y',v \rangle_{L^2}
 % \leq \\
% &\frac{2}{h} \left[ K_2 \| x'-y' \|_{L^2}(\|x'\|_{L^2} +
%   E_P\,\|X'\|_{L^2})\|v\|_{L^2} + K_1 \| x'-y' \|_{L^2}
%   \|v\|_{L^2}\right] \leq \\
% & \frac{2}{h} \left[ K_2 \| x'-y' \|_{L^2}2M\|v\|_{L^2} + K_1 \| x'-y' \|_{L^2}
%   \|v\|_{L^2}\right] \leq \\
\end{aligned}
\end{equation}
Since $\frac{d}{dt}K'_h(t^2) = 2K''_h(t^2) t \leq 2K_3$ by assumption (H1), $K'_h(t^2)$ is Lipschitz with
Lipschitz constant not larger that $2K_3$. Therefore,
\begin{equation}
\begin{aligned}
 &\left| K_h'\left( \|X-y\|^2_{L^2} \right)
      - K_h'\left( \|X-x\|^2_{L^2}\right) \right| \\ 
&\leq 2K_2 \big| \|X-y\|_{L^2} - \|X-x\|_{L^2}  \big| \\
      &\leq 2K_2 \|x-y\|_{L^2} \leq 
% 2K_3 C_p \|x' -y'\|_{L^2}       
2K_2 C_p\|x' -y'\|_{L^2}.
% \leq \\
% & \frac{K_2}{h}(2\|X\|_{L^2} +
%   \|x\|_{L^2} + \|y\|_{L^2}) \|x-y\|_{L^2} \leq \frac{K_2}{h}(2\|X\|_{L^2} +
%   M+N ) \|x'-y'\|_{L^2}.
\end{aligned}
\end{equation}
%
% where $C_p$ is the Poincare's constant. 
%
We have
\begin{equation}
\label{eq:part1}
\begin{aligned}
& E_P \,\left[ K_h'\left( \|X-y\|^2_{L^2} \right)
      - K_h'\left( \|X-x\|^2_{L^2}\right) \langle
    x'-X',v \rangle_{L^2} \right] \\
% &\le K_3 C_p \| x' -y'\|_{L^2}E_P\{\, \|x'\|_{L^2} + \|X'\|_{L^2} \} \|v\|_{L^2} \\
% & \le K_3 C_p (L+ N_1)  \| x' -y'\|_{L^2} \|v\|_{L^2}
&\leq 2K_2C_p \| x' -y'\|_{L^2} \left(\|x'\|_{L^2} + E_P\,\|X'\|_{L^2} \right) \|v\|_{L^2} \\
& \leq 2K_2C_p  \| x' -y'\|_{L^2}  (L+ N_1) \|v\|_{L^2}
\end{aligned}
\end{equation}
and
\begin{equation}
\label{eq:part2}
E_P\, K'_h\left( \|X-y\|^2_{L^2}\right)
\langle x'-y',v \rangle_{L^2} \leq K_2\|x'-y'\|_{L^2}\|v\|_{L^2}.
\end{equation}
By putting together equations \eqref{eq:part1} and \eqref{eq:part2} we then have
the following bound for equation \eqref{eq:main}:
\begin{equation}
\langle D p_h(x)'-D p_h(y)',v \rangle_{L^2} \leq
C(L) \|x'-y'\|_{L^2} \|v\|_{L^2},
\end{equation}
with $C(L)=2\left[K_3(L+ N_1) + K_2\right]$, which obviously implies $\|D
p_h(x)'-D p_h(y)'\|_{L^2} \leq C(L) \|x'-y'\|_{L^2}$.
\end{proof}
\begin{proof}[Proof of Lemma \ref{lemma:flowPointsTowardsBall}]
Lemma \ref{lemma:auxiliaryToL2GradientInH1} and Proposition
\ref{proposition:gradientIsInH1} allow us to write
\begin{equation}
\label{eq:boundedH1}
\begin{aligned}
&\langle D p_h(\pi_x(t)),\pi_x(t) \rangle_{H^1_0} = \langle D p_h(\pi_x(t))',\pi_x(t)'
\rangle_{L^2} \\
&= 2E_P \, K_h'\left( \|X-\pi_x(t)\|^2_{L^2}\right) \langle \pi_x(t)'-X',\pi_x(t)'
\rangle_{L^2} \\
& =2E_P \, K_h'\left( \|X-\pi_x(t)\|^2_{L^2}\right) \left(\|\pi_x(t)'\|^2_{L^2}-\langle X',\pi_x(t)'
\rangle_{L^2} \right)  \\
&\leq 2E_P \, K_h'\left( \|X-\pi_x (t)\|^2_{L^2}\right)
\|\pi_x (t)'\|^2_{L^2} \\
&- 2E_P \, K_h'\left( \|X-\pi_x (t)\|^2_{L^2}\right)\| X'\|_{L^2}\|\pi_x (t)'\|_{L^2},\\
\end{aligned}
\end{equation}
where the last inequality follows because (H2) guarantees that $K'_h(t^2) \leq
 0$. \\
For the first claim, assumption (H2) and $p_h(\pi_x(t)) \geq p_h(\pi_x(0)) \geq\delta$
imply
\begin{equation}
\begin{aligned}
&E_P \, K_h'\left( \|X-\pi_x(t)\|^2_{L^2}\right) \leq - E_P \, K_h\left(
  \|X-\pi_x(t)\|^2_{L^2}\right) \\
&=-p_h(\pi_x(t)) \leq -\delta.
\end{aligned}
\end{equation}
Thus, if $\|\pi_x (t)\|_{H^1_0}  \geq K_2N_1 /\delta$,
\begin{equation}
\langle D p_h(\pi_x(t)),\pi(t) \rangle_{H^1_0} \leq -2\delta \|\pi_x(t)\|^2_{H^1_0}+2K_2N_1 \|\pi_x(t)\|_{H^1_0} \leq 0.
\end{equation}
For the second part, equation \eqref{eq:boundedH1} gives
\begin{equation}
\begin{aligned}
&\langle D p_h(\pi(t)),\pi(t) \rangle_{H^1_0} \\
&\leq 2E_P \, K_h'\left( \|X-\pi_x(t)\|^2_{L^2}\right)\|\pi_x(t)'\|^2_{L^2} \\
&-2E_P \, K_h'\left( \|X-\pi_x(t)\|^2_{L^2}\right)\|
X'\|_{L^2}\|\pi_x(t)'\|_{L^2}\\
& \leq 2E_P \, K_h'\left( \|X-\pi_x(t)\|^2_{L^2}\right)\} \left( \| \pi_x(t) ' \|^2_{L^2} - M \|\pi_x(t)'\|_{L^2} \right).
\end{aligned}
\end{equation}
Thus, $\langle D p_h(\pi_x(t)),\pi_x(t) \rangle_{H^1_0} \leq 0$ as soon as $\|\pi_x(t)'\|_{L^2}=\|\pi_x(t)\|_{H^1_0}>M $.
\end{proof}
\begin{proof}[Proof of Proposition \ref{proposition:existenceuniqueness}]
Proposition \ref{proposition:gradientIsInH1} and Lemma \ref{lemma:gradientIsLocallyLipschitz}
 guarantee the existence and uniqueness of a local solution under the $H^1_0$ norm from the standard
 theory of ordinary differential equations. Some extra work is needed to extend
 the local solution to a global one. We provide a complete proof in three steps which
 builds on Theorem 3.10 of \cite{hunter2001applied} (their Theorem 3.10 holds more
 generally on Banach spaces, see for instance \citealp{schechter2004introduction}) and the authors' subsequent remark concerning the
 extension of the local solution to a global one.
 \begin{itemize}
\item[Step 1.] In this step, we show that if the solution $\pi_x(t)$ exists for
  any time interval $[0, T]$, then
 there exists $C_1>0$ such that $\|\pi_x(t)\|_{H^1_0} \leq C_1$.\\
If $p_h(x) =0$, then $x$ is a trivial local minimum of $p_h(x)$. As a result,
$Dp_h(\pi_x(0)) =0$ and $\pi_x(t) = \pi_x(0)$ for all $t$. Thus, in this case it
suffices to take $C_1 =R$. Suppose instead that $p_h(x) =\delta >0$. Consider
$g(t)=\|\pi_x(t)\|^2_{H^1_0}$. Clearly, $\frac{d}{dt}g(t)=2\left\langle
\pi_x(t),\frac{d}{dt}\pi_x(t) \right\rangle_{H^1_0} = 2\left\langle
\pi_x(t), D p_h(\pi_x(t)) \right\rangle_{H^1_0}$. Note that $g(0) \leq R^2$. 
Take $C_1= \max\{ R, K_2N_1/\delta\}$. Fix an
arbitrary $\epsilon>0$ and suppose that there exists $T'$ such that $0 \leq T' \leq T$ and
$g(T') \geq C_1^2+\epsilon$. Then, there must exist $0 \leq t^* \leq T'$
such that 
\begin{equation}
g'(t^*)=2\left\langle \pi_x(t^*),D p_h(\pi_x(t^*)) \right\rangle_{H^1_0} > 0
\end{equation}
and $g(t^*) \in (C_1^2,C_1^2+\epsilon)$. This is a contradiction because, by Lemma \ref{lemma:flowPointsTowardsBall}, 
if  $\|\pi_x(t^*)\|_{H^1_0} >K_2 N_1 /\delta $ then $\left\langle \pi_x(t^*),D
  p_h(\pi_x(t^*)) \right\rangle_{H^1_0} \leq 0$.
\item[Step 2.] Let $\pi_x :[0,T_1] \to H^1_0$ be the local solution of the
  ordinary differential equation $\pi'_x(t)=D p_h(\pi_x(t)) $ with $\pi_x(0)=x$.
Suppose that
$\|\pi_x(t)\|_{H^1_0} \leq C_1$ if $t \leq T_1$ . Given $C_2 >C_1$, we show that there exists $T_2 >0$ such that
 the solution can be uniquely extended to $\pi_x : [0, T_1+T_2] \to H^1_0$ with
 $\|\pi_x(t)\|_{H^1_0} \leq C_2$ if $t \leq T_1+T_2$. 
To see this, consider the ordinary differential equation $\frac{d}{dt} \phi(t) = Dp_h(\phi(t))$ with $\phi(0) = \pi_x(T_1)$. Note that
$\|\phi(0)\|_{H^1_0} = \|\pi_x(T_1)\|_{H^1_0} \leq C_1 <C_2$ by assumption. Also,
let $N>0$ be such that
\begin{equation}
\sup_{x\in B_{H^1_0}(0, M_2)} \|D p_h(x)\|_{H^1_0} \leq N.
\end{equation}
Now, by the Picard-Lindel\"of theorem on Banach spaces, if one takes $T_2= (C_2 -C_1)/ N$ then 
the solution $\phi$ exists on $[0, T_2]$ and $\phi(t) \in B_{H^1_0} (\pi_x(T_1), C_2 -C_1)$. 
Consider the extension $\pi_x(t)$ given by
\begin{equation}
\pi_x(t)=
\begin{cases}
  \pi_x(t)  & \text{ if }  t \leq T_1 \\
  \phi(t-T_1) & \text{ if } T_1 \leq t \leq T_1+T_2. 
\end{cases}
\end{equation}
The newly defined $\pi_x$ is well-defined and  continuous. Since 
\begin{equation}
\frac{d}{dt} \pi_x(t) = \frac{d}{dt} \phi(t-T_1) =Dp_h(\phi(t-T_1)) =D p_h(\pi_x(t))
\end{equation}
if $t\in [T_1,T_1+T_2]$, the new $\pi_x$ is an extension of the
solution. Furthermore, clearly $\pi_x(t) \in B_{H^1_0} (0, C_2)$ for $ t\in [0,
T_1+T_2]$. The uniqueness of the extended solution follows 
from the fact that $Dp_h$ is Lipschitz on $B_{H^1_0}(0, C_2)$.
 \item[Step 3.] Since $\|x\|_{H^1_0} \leq R$, by Picard's theorem there exists a local solution  
 $\pi_x(t) : [0, T_1] \to H^0_1$ and $\| \pi_x(t)\|_{H^1_0} \leq C_1$. Step 2 guarantees that 
 the solution can be uniquely extended to $[0, T_1+T_2] $.  Step 1 then implies that 
 such extended solution $\pi_x$ satisfies $\| \pi_x(t) \|_{H^1_0} \leq C_1$ for all $t\in [0, T_1+T_2]$. By Step 2 again,
 the extended solution $\pi_x$ can be extended again to the larger time interval $[0, T_1+2T_2]$ and, once again, by
 Step 1 the extended solution is entirely contained in the $H^1_0$ ball of radius
 $C_1$. By iterating this procedure, one sees that the unique solution $\pi_x$
 can be extended to all of $\mathbb R_+$ and $\|\pi_x(t) \|_{H^1_0} \leq C_1$ for
 all $t \geq 0$.
\end{itemize}
\end{proof}
\begin{proof}[Proof of Theorem \ref{theorem:solutionStaysInClosedBall}]
%
% Let $\pi_x(t) $ be given such that $\| \pi_x(t)\|_{H^1_0} \le C_1$ for all $t$ 
% and $\pi'_x(t) =D p_h(\pi_x(t))$.   
Since $p_h$ is a bounded functional and
both $Dp_h$ and $D^2p_h$ are bounded operators on $L^2$, 
it is clear that
\begin{equation}
\lim_{t\to \infty} \|Dp_h(\pi_x(t))\|_{L^2} = 0
\end{equation}
(see Lemma 7.4.4 in \citealp{jost2011riemannian}). Furthermore, since $\|\pi_x(t)\|_{H^1_0} \le C_1$ for all $t \geq 0$ and 
closed $H^1_0$ balls are compact with respect to the $L^2$ norm, there exist $\{\pi(t_k)\}^{\infty}_{k=1}$ such that 
 $\lim_{k \to \infty}\|\pi_x(t_k) - \pi_x(\infty)\|_{L^2} \to 0$ for some
 $\pi_x(\infty) \in L^2$. By the continuity of 
 $Dp_h :L^2 \to L^2$, one also has that $Dp_h(\pi_x(\infty))=0$.

Recall that by assumption (H4), all the non-trivial critical points of $p_h$ are isolated. Hence, for any
non-trivial 
critical point of $p_h$, one can find a $L^2$ neighborhood
around it in which there are no other critical points of $p_h$.
Let $\delta_1>0$ be the radius of such neighborhood around $\pi_x(\infty)$. Suppose now that the sequence $\{\pi_x(t)\}_{t \geq 0}$ does
not converge to $\pi_x(\infty)$ in the $L^2$ sense. Then, there exists
$\delta_2>0$ and a subsequence $\{\pi_x(s_k)\}_{k \geq 1}$ such that
$\|\pi_x(\infty)-\pi_x(s_k)\|_{L^2} \geq \delta_2$ for all $k \geq
1$. Without loss of generality, one can assume that $\|\pi_x(t_k) - \pi_x
(\infty)\|_{L^2} \leq \delta_1 / 3 $ and that $t_k < s_k < t_{k+1} $ for all $k$.
  But then, by the continuity of the path $\pi_x$, there exists
$r_k$ such that $ t_k \leq  r_k \le s_k$ and 
$\|\pi_x(\infty)-\pi_x(r_k)\|_{L^2} =  \min\{\delta_1,\delta_2\}/2$
  for all $k \geq 1$. Since $ \|\pi_x(r_k)\|_{H^1_0} \leq C_1 $, $\{\pi_x(r_k)
  \}_{k \geq 1}$ also has a subsequence which converges with respect to the $L^2$ norm as well. 
  Without loss of generality assume that 
   $\pi_x(r_k) \to\tilde \pi_x(\infty)$ in $L^2$ sense.
  By the continuity of $D p_h(x)$,  $\tilde \pi_x(\infty)$ is also a
  critical point of $p_h$. But then,
  $\|\pi_x(\infty)-\tilde \pi_x(\infty)\|_{L^2} = \min\{\delta_1,\delta_2\}/2
  < \delta_1$, which is a contradiction. This establishes the uniqueness of
  $\pi_x(\infty)$ and concludes the proof.
\end{proof}
\begin{proof}[Proof of Lemma \ref{lemma:criticalPointsInH1}]
By assumption, $Dp_h(x) = 2E_P\, K_h'(\|X-x\|^2_{L^2})(x-X) =0$ and $E_P \, K_h'(\|X-x\|^2_{L^2})\leq - E_P \, K_h(\|X-x\|^2_{L^2}) =-p_h(x)<0$.
Thus, 
\begin{equation}\label{criticalpoints}
  x= \frac{ E_P \, K_h'(\|X-x\|^2_{L^2})X}{ E_P \, K_h'(\|X-x\|^2_{L^2})}.
\end{equation}
Note that, by assumption (H2), $E_P \, K_h'(\|X-x\|^2_{L^2}) \leq -E_P \, K_h(\|X-x\|^2_{L^2}) < 0$. Therefore, it suffices to show that
$ E_P \, K_h'(\|X-x\|^2_{L^2})X \, \in H^1_0$. We have
\begin{equation}
\begin{aligned}
&\langle E_P \, K_h'(\|X-x\|^2_{L^2})X ,v'\rangle_{L^2}
=E_P \, K_h'(\|X-x\|^2_{L^2})\langle X ,v'\rangle_{L^2} \\
&=E_P \, K_h'(\|X-x\|^2_{L^2})\langle -X ',v\rangle_{L^2} \leq K_2 N_1 \|v\|_{L^2}.
\end{aligned}
\end{equation}
Thus, $E_P \, K_h'(\|X-x\|^2_{L^2})X  \in H^1_0$ by Lemma
\ref{lemma:auxiliaryToL2GradientInH1}. For the second claim of the Lemma, suppose that $\|X\|_{H^1_0} \leq M$ $P$-almost
surely. Then, any $x$ which is a non-trivial critical point of $p(x)$ satisfies equation \eqref{criticalpoints}.
As a result, for any $v\in C^{\infty}_{c}([0,1])$,
\begin{equation}
\begin{aligned}	
&\langle x ,v' \rangle_{L^2} =\frac{E_P K'_h(\|X-x\|^2_{L^2}) \langle X,v'\rangle_{L^2}}{E_P K'_h(\|X-x\|^2_{L^2})} \\
& =\frac{E_P\, K'_h(\|X-x\|^2_{L^2}) \langle -X',v\rangle_{L^2}}{E_P\, K'_h(\|X-x\|^2_{L^2})}\\
&\leq \frac{E_P\, K'_h(\|X-x\|^2_{L^2}) \| X\|_{H^1_0}\|v\|_{L^2}}{E_P \, K'_h(\|X-x\|^2_{L^2})}\\
&\leq M\|v\|_{L^2}.
\end{aligned}
\end{equation}
By Lemma \ref{lemma:auxiliaryToL2GradientInH1}, it follows that $\|x\|_{H^1_0} \leq M$
and the proof is complete.

\end{proof}
\begin{proof}[Proof of Lemma \ref{lemma:Sexample}]
In light of Proposition \ref{proposition:gradientIsInH1}, for the first claim it suffices to show that if $x \in S$, then
$Dp_h(x) \in S$. Note that $S$ is a closed subspace of $L^2$. As a result, there
exists another subspace $S^\perp \subset L^2$ which is the
orthogonal complement of $S$. Let $g\in S^\perp$, so that $\langle X,g\rangle_{L^2}=0$ almost surely.
Then,
\begin{equation}
 \langle Dp_h(x) , g\rangle_{L^2} = 2E_P \, K'_h (\|X-x\|^2_{L^2})\langle
 x-X,g\rangle_{L^2} =0,
\end{equation}
and thus $Dp_h(x) \in S$. The second claim is established in a similar way as in
Lemma \ref{lemma:criticalPointsInH1}.
\end{proof}
\begin{proof}[Proof of Lemma \ref{lemma:nonDegeneracyOfCriticalPoints}]
  By Lemma \ref{lemma:Sexample}, if $x^*$ is a non-trivial critical point then
   $x^* \in S$.
If one views $D^2 p_h(x^*)$ as a linear operator from $L^2$ to $L^2$, it is
sufficient to show that $D^2 p_h(x^*)$ is an isomorphism (i.e. a continuous map from $L^2$
to $L^2$ such that its inverse is also continuous). Note first that for any $v\in L^2$
\begin{equation}
\begin{aligned}
D^2p_h(x^*) (v)&= E_P \, \left[ 4 K_h'' (\|X-x^*\|^2_{L^2})\langle x^*-X ,v
  \rangle_{L^2} ( x^*-X ) \right. \\
&\left. + 2K'_h(\|X-x^*\|^2_{L^2})  v \right]. 
\end{aligned}
 \end{equation}
 Observe that
\begin{itemize}
\item[1.] If $v \in S$, then $D^2p_h(x^*)(v) \in S$. One can use a similar
  computation as in equation \eqref{derivativeequivalent} to show
that $D^2 p_h(x^*) (v) =D^2 \tilde{p_h} (\tilde x^*)(\tilde v)$, where $\tilde v$ is the vector in $R^d$ corresponding to $v$. 
\item[2.] Suppose $v \in S^\perp $. Since  $\langle x^* -X,v\rangle_{L^2}=0$ a.s.,
$D^2p_h(x^*) (v)\in S^\perp$. More specifically,
\begin{equation}
\label{eq:secondDerivativeOnOrtogonalComplement}
D^2p_h(x^*) (v)= 2 E_P \, K'_h( \|X-x^*\|^2_{L^2}) v .
\end{equation}
\end{itemize}
Thus, $S$ and $S^\perp$ are invariant subspaces of $D^2 p_h(x^*)$. In order 
 to see that $D^2p_h(x^*)$ is indeed an isomorphism, it is therefore enough to show
 that it is isomorphism on both $S$ and $S^\perp$ separately.
Under assumption (H4'), $p$ is a Morse density on $S_c$ and there exists $h>0$ small
enough so that $p_h$ is also a Morse function on the interior of $S_c$ (see
Remark \ref{remark:commentsOnMorseCondition}). Then, $x^*$ is in $S_c$ by
Proposition \ref{proposition:phIsMorse} and since $D^2p_h(x^*)$ is equivalent to $\nabla^2 \tilde p_h(\tilde x^*)$ (the
 Hessian of $\tilde p_h$ at $\tilde x^*$), for $h$ small enough $D^2p_h(x^*)$ is an isomorphism
 on $S$.
 Since $x^*$ is a non-trivial critical point of $p_h$, $p_h(x^*) =\delta >0$. By
 (H2), $E_P \, K_h'(\| X-x^*\|_{L^2})) \leq -\delta <0$. According to
 equation \eqref{eq:secondDerivativeOnOrtogonalComplement}, $D^2 p_h(x^*)$ 
 acts on $S^\perp$ by multiplying every vector in  $S^\perp$ by $2E_P \, K_h'(\|
 X-x^*\|_{L^2})$ and hence $D^2 p_h(x^*)$ is clearly an isomorphism on $S^\perp$.
 \end{proof}

\begin{proof}[Proof of Lemma \ref{lemma:equivalentnorms}]
Denote $T= - D^2 f(x^*)$ for simplicity. Then, $T$ is a positive definite isomorphism on $L^2$. Thus, there exists 
$C>0$ such that $\|T^{-1} \| \leq C$ where $\| \cdot \|$ here denotes the
operator norm. Also, it is straightforward to check that
 $T$ induces a well-defined inner product $\langle \cdot ,\cdot \rangle_T $ on
 $L^2$ by $\langle v ,w\rangle_T = \langle T v,w\rangle_{L^2}$. Now, for any $v
 \in L^2$ we have
\begin{equation}
\begin{aligned}
\|v\|_{L^2}^2 &= \langle v,v\rangle_{L^2} = \langle  T(T^{-1}v),v\rangle_{L^2} =  \langle  T^{-1}v,Tv\rangle_{L^2} = \langle T^{-1} v , v\rangle_T \\
& \le \|T^{-1} v \|_{T} \|v\|_{T}  = \|v\|_{T} \sqrt{\langle T^{-1} v, T^{-1}v \rangle_{T} } \\
&= \|v\|_{T} \sqrt{\langle T(T^{-1} v), T^{-1}v \rangle_{L^2} } \le \|v\|_{T}  \sqrt {\|T^{-1} v\|_{L^2} \|v\|_{L^2} }\\
&= \|v\|_{T} \sqrt{C} \|v\|_{L^2}.
\end{aligned}
\end{equation}
This implies that $\|v\|_{L^2}^2 \le C \|v\|^2_{T}$, and thus
\begin{equation}
\begin{aligned}
\sup_{\|v\|_{L^2} =1}Df^2(x^*) (v,v) &=\sup_{\|v\|_{L^2}
  =1}-T(v,v) \\ 
&=\sup_{\|v\|_{L^2} =1} -\| v\|^2_{T} \le -1/C.
\end{aligned}
\end{equation}
Therefore, by taking $\delta =1/C$ the claim of the Lemma follows.
\end{proof}
\begin{proof}[Proof of Lemma \ref{lemma:maxareclose}]
Let $\delta =\delta (x^*_2)$. The proof is in three steps.
\begin{itemize}
\item[Step 1.] Suppose that $\eta_1 \le \delta^2/8\beta_3$. Then, if $\epsilon=\delta/2\beta_3$, 
the solution of the initial value problem $\pi'_1(t) = Df_1 (\pi _1(t))$, with
$\pi_1(0) =x^*_2$  is contained in $B_{L^2} (x^*_2 ,\epsilon)$. In fact, suppose that the trajectory $\pi_1$ is not contained in $B_{L^2} (x^*_2 ,\epsilon)$. Since $\pi_1(0) =x^*_2$,
there must exist $t_0 >0$ such $\|\pi_1(t_0) -x^*_2\|=\epsilon$. Denote
$\pi(t_0) =x_0$. Then since $Df_2(x_2^*)=0$, a Taylor expansion implies that
\begin{equation}
\begin{aligned}
& f_1(x_0) \le f_1(x^*_2) +\langle Df_1 (x^*_2) ,x_0-x^*_2\rangle_{L^2} \\
&+ \frac{1}{2} D^2f_1(x^*_2) (x_0-x^*_2, x_0-x^*_2) \\
 & +\frac{1}{6} \beta_3 \|x_0-x^*_2\|_{L^2}^3\\
 & \le f_1(x^*_2) + \langle Df_2 (x^*_2) ,x_0-x^*_2\rangle_{L^2}   \\
& +\|Df_1 (x^*_2) -Df_2(x^*_2) \|_{L^2}  \|x_0 -x^*_2\|_{L^2}  \\
 & + \frac{1}{2} D^2f_2(x^*_2) (x_0-x^*_2, x_0-x^*_2) \\
& +\frac{1}{2} \|D^2f_1(x^*_2)  -D^2f_2(x^*_2) \|_{L^2} \| x_0 -x^*_2\| ^2_{L^2}\\
 &+\frac{1}{6} \beta_3 \|x_0-x^*_2\|_{L^2}^3\\
 &\le f_1(x^*_2) + \eta_1 \epsilon -\frac{1}{2}\delta\epsilon^2 +\frac{1}{2}\eta_2 \epsilon^2 +\frac{1}{6} \beta_3\epsilon^3\\
 & \le f_1(x^*_2) + \frac{1}{4} \delta \epsilon^2 -\frac{1}{2}\delta\epsilon^2 +\frac{1}{16} \delta \epsilon^2 +\frac{1}{12} \delta\epsilon^2\\
 &< f_1(x^*_2),
 \end{aligned}
\end{equation}
which is a contradiction because $f_1(\pi_1(t))$ is an non-decreasing function of $t$.
\item[Step 2.] By condition (C2), $\pi_1$ admits a convergent subsequence in $L^2$. Thus there is a subsequence $\{t_k\}_{k=1}^{\infty}$ and a critical point $x^*_1$   
such that $\|\pi_1(t_k) - x^*_1\|_{L^2} \to 0$ as $k \to \infty$ and $x^*_1 \in  B_{L^2}(x^*_2,\epsilon)$. In order to show that 
$x^*_1$ is a non-degenerate  local maximum in $B_{L^2}(x^*_2,\epsilon)$, consider $\eta_2 \le\delta/8$. 
Given any $\|u\|_{L^2} =1$, for any $x \in B_{L^2}(x^*_2,\epsilon)$ one has
\begin{equation}
\label{eq:testlocalmax}
\begin{aligned}
& D^2f_1(x) (u, u) \\
& \le D^2f_1(x^*_2) (u,u)  + |D^2f_1(x^*_2) (u,u)  -D^2f_1(x) (u,u) | \\
& \le D^2f_2(x^*_2) (u,u)  + |D^2f_2(x^*_2) (u,u) -D^2f_1(x^*_2) (u,u) | \\
& +\beta_3\|x^*_2 -x\|_{L^2}\\
& \le -\delta +\eta_2 + \beta_3\epsilon  =-\frac{3}{8}\delta.
\end{aligned}
\end{equation}
Therefore $\sup_{\|u\|_{L^2}=1} D^2f_1(x) (u,u) \le -3\delta /8$ and $Df_1(x^*_1)$ is negative definite. If one views $-Df_1(x^*_1)$ as a linear operator from $L^2$ to $L^2$, then
by the Lax-Milgram theorem, it is an isomorphism and hence $x_1^*$ is a
non-degenerate local maximum. Moreover, $x^*_1$ is the unique maximum in
$B_{L^2}(x^*_2,\epsilon)$: suppose that $y^*_1$ is another local maximum of $f_1$ in $B_{L^2}(x^*_2,\epsilon)$;
then, $Df_1(x^*_1)= 0$ and $Df_1(y^*_1)=0$, and by equation
\eqref{eq:testlocalmax},  $\sup_{\|u\|_{L^2} =1}D^2f_1(y^*_1)( u, u) \le
-3\delta /8$. A Taylor expansion shows that 
\begin{equation}
\begin{aligned}
 f_1(x^*_1) &\le f_1(y^*_1) + \frac{1}{2} D^2f_1(y^*_1) (x^*_1-y^*_1,
 x^*_1-y^*_1) \\
&+\frac{1}{6} \beta_3 \|x^*_1-y^*_1\|^3\\
 & \le f_1(y^*_1) - \frac{3}{16} \delta \|x_1^*-y^*_1\|^2 +\frac{1}{6} \epsilon\beta_3 \|x^*_1-y^*_1\|^2 \\
  &\le f_1(y^*_1) - \frac{5}{48} \delta \|x_1^*-y^*_1\|^2
 \end{aligned}
\end{equation}
and by symmetry, $f_1(y^*_1) \le f_1(x^*_1)  - \frac{5}{48} \delta
\|x_1^*-y^*_1\|^2$, which is a contradiction unless $y^*_1 = x^*_1$.
\item[Step 3.] Now it is only left to show that $\|x^*_1 -x^*_2\|_{L^2} \le
  C\eta_1$. Since $Df_1(x)$ is a twice continuously differentiable function, a
  Taylor expansion around $x^*_2$ allows us to write
\begin{equation}
\begin{aligned}
  &\langle Df_1(x_2^*),\cdot \rangle_{L^2} = \langle  Df_1(x_1^*), \cdot \rangle_{L^2}  + D^2f_1(x_1^*)( x_2^* -x_1^*,\cdot) \\
&+\int_0^1 \frac{1}{2} D^3f_1 (x^*_1 +s( x^*_2 -x^*_1)) ( x_2^* -x_1^*, x_2^* -x_1^*, \cdot)\,ds.
 \end{aligned}
\end{equation} 
Note that one can replace $Df_1(x^*_1)$ by $Df_2(x^*_2)$ as both of them are $0$. Apply this identity to $x^*_2-x^*_1$, then
\begin{equation}
\begin{aligned}
&\langle Df_1(x_2^*) -  Df_1(x_2^*) , x^*_2 -x^*_1 \rangle_{L^2}  \\
&\le D^2f_1(x_1^*)( x_2^* -x_1^*,   x_2^* -x_1^*) +\frac{1}{2} \beta_3 \| x^*_1 -x^*_2\|^3_{L^2}\\
&\le  -\frac{3}{8}\delta \| x^*_1 - x^*_2\|^2 +\frac{1}{4} \delta \| x^*_1 - x^*_2\|^2 \\
& \le -\frac{1}{8} \delta \| x^*_1 - x^*_2\|^2.
 \end{aligned}
\end{equation}
This is equivalent to
\begin{equation}
  \| x^*_1 - x^*_2\|^2 \le\frac{8}{\delta} \|Df_1 (x^*_2) - Df_2(x^*_2) \| \|
  x^*_1-x^*_2\|.
\end{equation}
Taking $C= 8$ completes the step.
\end{itemize}
\end{proof}

\begin{proof}[Proof of Proposition \ref{proposition:confidenceMarking}]
First of all note that  since $P(\|X\|_{H^1_0}\le M) =1 $, lemma 5 ensures that 
all the non trivial critical points of  $f_1(x)=p_h(x)$and $f_2(x)=\hat p_h(x)$ are contained in 
$B_{H^1_0} (0,M)$. Let 
\begin{equation}
\eta_1 = \sup_{x \in B_{H^1_0}(0,M)} \| D\hat p_h(x) - Dp_h(x) \|_{L^2}
\end{equation}
and 
\begin{equation}
 \eta_2 = \sup_{x \in B_{H^1_0}(0,M)} \| D^2\hat p_h(x) - D^2p_h(x) \|.
\end{equation}
Consider the events $A=\{\eta_1 \leq C_1(\alpha)\}$ and $B=\{ \eta_2 \leq
 C_2(\alpha)\}$ where $C_1(\alpha)$ and $C_2(\alpha)$ are defined in Display
\ref{table:confidenceMarking}.
% It is easily seen using the exponential
% inequality of the Corollary of Lemma 4.3 in \cite{yurinskiui1976exponential}
% that
We can then use the uniform exponential inequalities on the first and second
derivatives of Lemma \ref{lemma:expInequalityFirstDerivative} and Lemma \ref{lemma:expInequalitySecondDerivative} of $\hat p_h$ to ensure
 $P((A \cap B)^c) =P(A^c+B^c)\le  P(A^c) +P(B^c)  \le \alpha
$ for $n$ large enough (which will be justified later in the proof). For now, under the event $A \cap B$, for each point $\hat x^*$ marked by
the algorithm of Display \ref{table:confidenceMarking}, i.e, $\hat x^* \in \mathcal{ \hat R}$ we have
\begin{equation}
\delta^2 \geq 8\beta_3C_1(\alpha) \geq 8\beta_3  \eta_1
\end{equation}
and
\begin{equation}
\delta \geq 8C_2(\alpha) \geq 8 \eta_2,
\end{equation}
hence the assumptions of Lemma \ref{lemma:maxareclose} are
satisfied. Furthermore, Lemma \ref{lemma:maxareclose} ensures that the ball
$B_{L^2}(\hat x^*,\delta(\hat x^*)/(2\beta_3) )$ contains a unique non-degenerate local mode $x^*$ of $p_h$ and that $\|\hat x^* -x^*\|_{L^2} \le 8\eta_1/\delta(\hat x^*)\le 8C_1(\alpha)/\delta(\hat x^*) $ under the event $A
\cap B$.

To justify $P(A^c)=P(\eta_1 \ge C_1(\alpha)) \le\alpha/2$, consider
the inequality of Lemma \ref{lemma:expInequalityFirstDerivative}. We have 
\begin{equation}
\begin{aligned}
&P \left( \sup_{x \in B_{H_0^1}(0,M)} \| D\hat p_h(x) - Dp_h(x)\|_{L^2}  \geq \epsilon\right) \\
&\leq C \exp \left(-\frac{4n\epsilon^2}{25K_1^2}+\frac{10MK_2}{\epsilon}\right).
\end{aligned}
\end{equation}
Let $a=4/(25K_1^2)$, $b=10MK_2$, $d=\log\left(\frac{\alpha}{2C}\right)<0$.
Take
\begin{equation}
\epsilon=\left(\frac{b}{a}\right)^{\frac{1}{3}}n^{-\frac{1}{3}} + \left(\frac{-d}{a}\right)^{\frac{1}{2}}n^{-\frac{1}{2}}.
\end{equation}
Then,
\begin{equation}
\begin{aligned}
&-\frac{4n\epsilon^2}{25K_1^2}+\frac{10MK_2}{\epsilon}\\
&=-an\epsilon^2+\frac{b}{\epsilon} \leq -an \left(
  \left(\frac{b}{a}\right)^{\frac{1}{3}}n^{\frac{1}{3}} +\left(\frac{-d}{a}
  \right)^{\frac{1}{2}}n^{\frac{1}{2}}\right)^2+\frac{b}{\left( \frac{b}{a}
  \right)^{\frac{1}{3}}n^{-\frac{1}{3}}} \\
&\leq -an \left(\left(\frac{b}{a}\right)^{\frac{2}{3}}n^{-\frac{2}{3}} +
  \frac{-d}{a}n^{-1}\right) + a^{\frac{1}{3}} b^{\frac{2}{3}}n^{\frac{1}{3}}\\
&=-a^{\frac{1}{3}} b^{\frac{2}{3}}n^{\frac{1}{3}} + d + a^{\frac{1}{3}}
b^{\frac{2}{3}}n^{\frac{1}{3}} = d =\log\left( \frac{\alpha}{2C}\right).
\end{aligned}
\end{equation}
With this particular choice of $\epsilon=C_1(\alpha)$ it then follows that
\begin{equation}
\begin{aligned}
&P \left( \sup_{x \in B_{H_0^1}(0,M)} \| D\hat p_h(x) - Dp_h(x)\|_{L^2} \geq
  \epsilon \right) \\
&\leq C e^{d} = C\frac{\alpha}{2C} = \alpha/2.
\end{aligned}
\end{equation}
An almost identical argument is used to justify $P(B)=P( \eta_2 \ge C_2(\alpha)) \le \alpha/2$.

% \begin{aligned}
% -a  \frac{-\log \alpha}{d} \right)^{\frac{2}{3}} n^\frac{1}{3} + b  \frac
% {d}{-\log \alpha} \right)^{\frac{1}{3}}n^{\frac{1}{3}} &=
% \end{aligned}

%
% \begin{equation}
% \epsilon=-\frac{8\eta_1}{\sup_{\|u\|_{L^2}=1} D^2 \hat p_h(x^*) ( u, u)}.
% \end{equation}
%
\end{proof}

\begin{proof}[Proof of Proposition \ref{prop:type2ofmarking}]
Taking  $f_1(x)  =\hat p_h(x) $ and  $f_2(x)=p_h(x)$, the goal is to apply Lemma \ref{lemma:maxareclose} for all non-trivial critical points of $p_h$.
It is worth to mention that, under the given assumption, $\mathcal R$ is a finite set and $\mathcal R \subset B_{H^1_0}(0,M)$. As a result, there exists  a $\gamma$ such that
\begin{equation} 
-\gamma:=\sup_{x^*\in \mathcal C}\sup_{\|u\|_{L^2} =1} D^2 p_h(x^*) (u,u) <0.
\end{equation}
According to Lemmas
 %\ref{lemma:expInequality} ,
\ref{lemma:expInequalityFirstDerivative} and \ref{lemma:expInequalitySecondDerivative}, for $l=1,2$ there exist 
 constants $0<H_l<\infty$ and $0<h_l<\infty $ depending only on $K_1$, $K_2$ and $K_3$ and $M$ such that
\begin{equation}
\begin{aligned}
\label{boundfortype2}
P\left(\sup_{x\in B_{H^1_0}(0,M)} \|D^l\hat p_h(x) -D^l p_h(x) \|\ge \frac{H_l} {n^{1/3}}\right) \le C\exp\left(-h_ln^{1/3}\right).
\end{aligned}
\end{equation}
Let $\eta_l$, $l=1,2$ be defined as in (C3).  Let $F_n:=\{\eta_1\le \frac{H_1}{n^{1/3}}\} \cap \{\eta_2\le \frac{H_2}{n^{1/3}}\}$. Then, $P(F_n)\to 1$. The rest of the argument follows by assuming that $F_n$ holds.\\ \\
Suppose that, for large $n$, $H_1n^{-1/3}\le \gamma^2/(8\beta_3)$ and $H_2 n^{-1/3} \le \gamma /8$.
 Then, for all $x^* \in\mathcal R$, one has 
 \begin{equation}
 \begin{aligned}
 &\eta_1 \le H_1n^{-1/3}\le \gamma^2/(8\beta_3)  \le \delta(x^*)^2/(8\beta_3) \\
 &\eta_2\le H_2 n^{-1/3} \le \gamma /8 \le\delta(x^*)/8,
 \end{aligned}
 \end{equation}
 where as before 
 \begin{equation} 
 -\delta(x^*):=\sup_{\|u\|_{L^2} =1} D^2 p_h(x^*) (u,u) <0.
  \end{equation}
One can apply Lemma \ref{lemma:maxareclose} to all $x^*$ to conclude that there exists a $\hat x^*$ such that
\begin{itemize}
\item[1.] $\hat x^*$ is the unique local maximum of 
 $\hat p_h $ in $B_{L^2}(x^*,\delta(x^*)/(2\beta_3))$;
 \item[2.] 
 $\delta(\hat x^*):=-\sup_{\|u\|_{L^2} =1} D^2\hat p_h(\hat x^*) (u,u) \ge 3\delta(x^*)/8 \ge 3\gamma/8$; 
 \item[3.] $\|x^*-\hat x^*\|_{L^2}\le 8\eta_1/\delta(x^*)$.
 \end{itemize}
The following three steps complete the proof.
\begin{itemize}
\item[step 1.]
In this step, one shows that $\hat x^* \in \hat {\mathcal R}$. 
 According to item 2. in the first paragraph, $-\delta(\hat x^*):=\sup_{\|u\|_{L^2} =1} D^2\hat p_h(\hat x^*) (u,u) \le -3\gamma/8$. Thus,
  \begin{equation}
 -\sup_{\|u\|_{L^2} =1} D^2\hat p_h(\hat x^*) (u,u) \ge 3\gamma /8 \ge \max\{ \sqrt{8 \beta_3 C_1(\alpha)}, 8 C_2(\alpha)\}
\end{equation}
because both $C_1(\alpha)$ and $C_2(\alpha )$ are of order $O(n^{-1/3})$.  
\item[step 2.] One shows that $\Phi(\hat x^*) =x^*$, where  $\Phi$ is defined in equation \eqref{eq:mapbetweenmax}.
Then, according to equation \eqref{eq:mapbetweenmax}, it suffices to show that
\begin{equation}
x^* \in B_{L^2}(\hat x^*,\delta(\hat x^*)/(2\beta_3))\cap B(\hat x^*,\log(n) C_1 (\alpha)/\delta(\hat x^*)).
\end{equation}
From item 3. in the first paragraph, $\|\hat x^*-x^*\|\le 8\eta_1/\delta(x^*)$. 
Thus it suffices to show that 
\begin{equation}
\begin{aligned}
&B(\hat x^*, 8\eta_1/\delta(x^*)) \subset \\
& B_{L^2}(\hat x^*,\delta(\hat x^*)/(2\beta_3))\cap B(\hat x^*,\log(n) C_1 (\alpha)/\delta(\hat x^*)).
\end{aligned}
\end{equation}
This is equivalent to
\begin{equation}
\label{powerineq}
8\eta_1/\delta(x^*) \le \delta(\hat x^*)/(2\beta_3)  \ 
 \text {and} \ 
  8\eta_1/\delta(x^*)\le \log(n) C_1 (\alpha)/\delta(\hat x^*).
\end{equation}
The first inequality of \eqref{powerineq} is clear because  
\begin{equation}
\delta(\hat x^*) \ge 3\gamma/8 \ \ \text{and}  \ \  8 \eta_1/\delta(x^*)\le 8H_1n^{-1/3}/\gamma =O(n^{-1/3}).
\end{equation}
 The second one holds for large $n$ because 
 \begin{equation}
8\eta_1/\delta(x^*)\le 8H_1n^{-1/3}/\gamma
\end{equation}
while
\begin{equation}
\begin{aligned}
&\log(n) C_1 (\alpha)/\delta(\hat x^*)  \\
&\ge \log(n)C_1(\alpha)/\beta_2=C(\alpha,K_1,K_2,K_3,M)n^{-1/3}\log(n).
\end{aligned}
\end{equation}
\item[step 3.] To complete the argument, it suffices to show that if $\Phi(\hat y^*) = x^*$ for some $\hat y^*\in \hat {\mathcal R}$, then $\hat y^*=\hat 
x^*$. Since $\hat y^*\in \hat  {\mathcal R}$, by the algorithm of Display \ref{table:confidenceMarking},
\begin{equation} 
\delta(\hat y^*)\ge \max\{\sqrt{8\beta_3 C_1(\alpha)}, 8C_2(\alpha)\}.
\end{equation}
Thus, $\delta (\hat y^*) \ge c(\alpha,M,K_1,K_2,K_3)n^{-1/6} $ for some $c(\alpha,MK_1,K_2,K_3)>0$ independent of $n$. As a result, since $\Phi(\hat y^*) =x^*$,
\begin{equation}
\begin{aligned}
&\|x^* -\hat y^*\|_{L^2} \le \log(n) C_1(\alpha) /\delta(\hat y^*) =O (\log(n)n^{-1/6}).
\end{aligned}
\end{equation}
Then, for large $n$
\begin{equation}
\|\hat y^*-x^*\|_{L^2} \le \gamma/(2\beta_3) \le \delta(x^*)/(2\beta_3)
\end{equation}
and therefore $\hat y^* \in B(x^* ,\delta(x^*)/(2\beta_3))$. According to item 1. in the first paragraph, $\hat x^*$ 
is the unique local maximum of $\hat p_h$ in $B(x^* ,\delta(x^*)/(2\beta_3))$. It thus follows that $\hat y^*=\hat x^*$. 
\end{itemize}
\end{proof}

 \begin{proof}[Proof of Lemma \ref{lemma:connection}]
 We discuss the case $l=1$. Only the constants differ in the remaining cases. For any $x \in L^2$
 \begin{equation}
 \begin{aligned}
 &\| D \hat  p_h(x) -D \tilde p_h(x)\|\\
  \le& \frac{2}{n}\sum_{i=1}^{n}  \left \| K'_h( \| X_i -x\|^2)_{L^2}(x-X_i) - K'_h(\|\tilde X_i-x\|_{L^2}) (x-\tilde X_i)\right\|   \\
   \le& \frac{2}{n} \sum^n_{i=1} K_2\| X_i-x -(\tilde X_i -x)\| 
  \\
  =& \frac{2}{n} \sum^n_{i=1} K_2\| X_i -\tilde X_i \| 
 \end{aligned}
 \end{equation} 
Thus,
 \begin{equation}
 \begin{aligned}
 & E \left( \sup_{x\in L^2}\| D \hat  p_h(x) -D \tilde p_h(x)\| \ \big|X_1,\ldots,X_n \right)\le 2K_2\phi(m),
 \end{aligned}
 \end{equation} 
where $\phi(m)$ does not depend on $X_i$. Therefore, this implies   
 \begin{equation}
 \begin{aligned}
 & E \left( \sup_{x\in L^2}\| D \hat  p_h(x) -D \tilde p_h(x)\| \right)\le 2K_2\phi(m).
 \end{aligned}
 \end{equation}
 As a result,
  \begin{equation}
  \begin{aligned}
P \left(\sup_{x \in L^2}\|D \hat p_h(x) -D \tilde p_h(x) \|\ge \epsilon \right) \le \frac{2K_2\phi(m)}{\epsilon}.
  \end{aligned}
  \end{equation}
 \end{proof}

 \begin{proof}[Proof of Corollary \ref{corollary:consistenttildep}] 
 The argument for the first part is almost the same as the one in Proposition 
   \ref{proposition:confidenceMarking}, except that in this case one makes use of the fact that 
   \begin{equation}
   \begin{aligned} &P\left( \sup_{x\in B_{H^1} (0,M) } \|Dp_h(x) -D\tilde p_h(x) \|\ge\tilde C_1(\alpha) \right)\\
\le &P\left( \sup_{x\in B_{H^1} (0,M) } \|Dp_h(x) -D\hat  p_h(x) \|\ge C_1(\alpha/2) \right) \\
+&P\left( \sup_{x\in B_{H^1} (0,M) } \|D\hat p_h(x) -D\tilde  p_h(x) \|\ge \frac{8K_2\phi(m)}{\alpha} \right) \\
\le& \alpha/4 +\alpha/4 =\alpha/2,
\end{aligned}
\end{equation}
and that
 \begin{equation}
   \begin{aligned} 
   &P\left( \sup_{x\in B_{H^1} (0,M) } \|D^2p_h(x) -D^2\tilde p_h(x) \|\ge\tilde C_2(\alpha) \right) \le \alpha/2. 
\end{aligned} 
\end{equation}
 The argument for the second part is the same as the one in Proposition \ref{prop:type2ofmarking}, except that equation \eqref{boundfortype2} becomes
 \begin{equation}
\begin{aligned}
&P\left(\sup_{x\in B_{H^1_0}(0,M)} \|D^l\tilde p_h(x) -D^l p_h(x) \|\ge \frac{H_l} {n^{1/3}} +\sqrt{\phi(m)}\right) \\
\le &P\left(\sup_{x\in B_{H^1_0}(0,M)} \|D^l\tilde p_h(x) -D^l \hat p_h(x) \|\ge \sqrt{\phi(m)}\right) \\
+ &P\left(\sup_{x\in B_{H^1_0}(0,M)} \|D^l p_h(x) -D^l \hat p_h(x) \|\ge \frac{H_l} {n^{1/3}}\right) \\
\le &C\left(\exp(-h_ln^{1/3}) +\sqrt{\phi(m)}\right).
\end{aligned}
\end{equation}
 \end{proof}

%
% \begin{proof}[Proof of Corollary \ref{corollary:testofmax}]
% Fix $u$ such that $\|u\|_{L^2} =1$, then 
% \begin{equation}
% \begin{aligned}
% & D^2 f_1(x^*_1) (u , u) \\
% &\le D^2 f_2 (x^*_1)(u, u) + |D^2 f_2 (x^*_1)(u, u) -D^2 f_1 (x^*_1)(u, u)| \\
% & \le D^2 f_2 (x^*_1)(u, u) + \eta_2 \|u\|^2\\
% & \le D^2 f_2 (x^*_2)(u, u) + | D^2 f_2 (x^*_2)(u, u) -D^2 f_2 (x^*_1)(u, u)| +\eta_2 \\
%  & \le \sup_{\|u\|_{L^2}=1} D^2 f_2(x^*_2) ( u, u)  + \beta_3 \|x^*_1-x^*_2\| \|u\|_{L^2}^2+\eta_2 \\
%  & \le \sup_{\|u\|_{L^2}=1} D^2 f_2(x^*_2) ( u, u)  + \frac{8}{\delta}\beta_3 \eta_1+\eta_2.
%  \end{aligned}
% \end{equation}
% The corollary follows by taking the supremum over all $u$ such that $\|u\|_{L^2}=1$.
% \end{proof}

% \section*{Acknowledgements}
% And this is an acknowledgements section with a heading that was produced by the
% $\backslash$section* command. Thank you all for helping me writing this
% \LaTeX\ sample file. See \ref{suppA} for the supplementary material example.

\section{Additional Results}
\label{section:additionalResults}
%
% The discussion in the previous sections are focused on the population functional $p_h$. 
% Let $$\bar p_h(x) = \frac{1}{n} \sum^{n}_{i=1} K_h (\|x-X_i\|^2_{L^2}).$$
% So $E_P (\bar p_h(x)) = p_h(x)$. Since the clustering method is based on the gradient flows,
% the supremum difference between $D^l \bar p_h(x)$ and $D^lp_h(x)$ will be considered
% when  $x$ is contained in a fixed $H^1_0$ ball.\\ \\
% In the first subsection,  the discussion will be based on noise free data. In the second subsection,
% a mild assumption will be taken in order to handle the noisy observations  
% \subsection{The Noise free case}
% For a random function $X$, assume that the data $\{X_i\}_{i=1}^n$ are observed as noise free functions.

\begin{lemma}
\label{lemma:expInequality}
Under the assumptions (H1), (H2), and (H3), 
\begin{equation}
\begin{aligned}
& P \left(\sup_{x\in B_{H^1_0} (0, M)} | \hat p_h(x) - p_h(x)| \ge  \epsilon\right) \\
& \leq C \exp\left(-\frac{32n\epsilon^2}{25K^2_0}+\frac{10MK_1}{\epsilon}\right)
\end{aligned}
\end{equation}
for $\epsilon$ sufficiently small.
\begin{proof}
By Chapter 7 of \cite{shiryayev2013selected}, the covering number $N_{\epsilon}$ of the ball $B_{H^1_0} (0, M)$ satisfies 
$N_{\epsilon} \le C\exp \left( \frac{M}{\epsilon}\right) $.
Let $\epsilon '=\epsilon/(10K_1)$. For a fixed radius $M$, pick $\{x_k\}_{k=1}^{N_{\epsilon'}}$ such that if $x\in B_{H^1_0} (0, M)$ then there exists $\|x_k-x\|_{L^2}\le\epsilon'$.
Note that for any fixed $x \in B_{H^1_0} (0,M)$,
\begin{equation}
\begin{aligned}
|\hat p_h (x) - \hat p_h(x_k)| &=\left| \frac{1}{n} \sum^{n}_{i=1} K_h(\|x-X_i\|^2_{L^2} ) -K_h(\|x_k - X_i\|^2_{L^2} )\right|\\
& \le \frac{1}{n} \sum_{i=1}^{n} \left| K_h(\|x-X_i\|^2_{L^2} ) -K_h(\|x_k - X_i\|^2_{L^2} )\right| \\
& \le \frac{1}{n}  \sum_{i=1}^{n} K_1 \|x-x_k\|_{L^2} =\frac{\epsilon}{10}.
\end{aligned}
\end{equation}
Thus, $| p_h (x) - p_h(x_k)| \le E_P|\hat p_h (x) -\hat p_h(x_k)| \le \frac{\epsilon}{10}$.
Since for any $x$,  
\begin{equation}
\begin{aligned}
&| \hat p_h(x) - p_h(x)|  \\ 
&\le | \hat p_h(x) -\hat p_h(x_k) | +| \hat p_h(x_k) - p_h(x_k)|+|
p_h(x_k) -p_h(x)|\\
&\le  | \hat p_h(x_k) - p_h(x_k)| +\frac{\epsilon}{5},
\end{aligned}
\end{equation}
it follows that
\begin{equation}
\begin{aligned}
&P \left(\sup_{x\in B_{H^1_0} (0, M)} | \hat p_h(x) -p_h(x) | \ge
  \epsilon\right) \\
&\le P \left(\sup_{1\le k \le N_{\epsilon'}} | \hat p_h(x_k) -p_h(x_k)| \ge \frac{4\epsilon}{5}\right)\\
& \le N_{\epsilon'}  P\left(  | \hat p_h(x_1) - p_h(x_1)| \ge \frac{4\epsilon}{5}\right) \\
& \le C\exp\left(\frac{10MK_1}{\epsilon}\right) \exp \left(-\frac{32n\epsilon^2}{25K_0^2}\right),
\end{aligned}
\end{equation}
where the last step uses Hoeffding's inequality.
\end{proof}
\end{lemma}

\begin{lemma}
\label{lemma:expInequalityFirstDerivative}
Under the same assumptions of the last Lemma, for $\epsilon$ sufficiently small,
\begin{equation}
\begin{aligned}
& P \left(\sup_{x\in B_{H^1_0} (0, M)} \| D\hat p_h(x) - Dp_h(x)\|_{L^2} \ge
  \epsilon\right) \\
& \le C \exp\left(-\frac{4n\epsilon^2}{25K_1^2}+\frac{10M
    K_2}{\epsilon}\right ).
\end{aligned}
\end{equation}
\begin{proof}
The proof is very similar to that of the previous Lemma. Notice first that
\begin{equation}
\begin{aligned}
&\left\| K'_h(\|x-X_i\|^2_{L^2} )(x-X_i) -K'_h(\|x_k - X_i\|^2_{L^2}
  )(x_k-X_i)\right\|_{L^2} \\
& \le K_2\|x-x_k\|_{L^2}.
\end{aligned}
\end{equation}
By taking $\epsilon'=\epsilon/(10K_2)$ and using the same argument of the
previous Lemma, we have
\begin{equation}
\begin{aligned}
& P \left(\sup_{x\in B_{H^1_0} (0, M)} \| D\hat p_h(x) - Dp_h(x)\|_{L^2} \ge \epsilon\right) \\
& \le N_{\epsilon'}P\left(  \| D\hat p_h(x_1) -Dp_h(x_1)\|_{L^2} \ge
  \frac{4\epsilon}{5}\right).
\end{aligned}
\end{equation}
In oder to proceed, we need an Hoeffding-type inequality for Hilbert spaces. Specifically, 
using Lemma \ref{lemma:derivativesOfp}, one has
\begin{equation}
\begin{aligned}
Dp_h(x_1)=\frac{2}{n}\sum_{i=1}^nE_P K_h'
(\|X_i-x_1\|^2_{L^2})(x_1-X_i).
\end{aligned}
\end{equation}
Now, if one denotes $Z_i$ as
\begin{equation}
Z_i = 2K_h' (\|X_i-x_1\|^2_{L^2})(x_1-X_i) -2E_P K_h' (\|X_i-x_1\|^2_{L^2})(x_1-X_i),
\end{equation}
then 
\begin{equation}
\begin{aligned}
D\hat p_h(x_1) -Dp_h(x_1)  =\frac{1}{n}\sum_{i=1}^n Z_i.
\end{aligned}
\end{equation}
Thus, $E_P Z_i =0$ as a $L^2 $ function and $\|Z_i\|_{L^2} \le 4K_1$.\\
Finally, by using the exponential inequality of the Corollary of Lemma 4.3 in \cite{yurinskiui1976exponential},
\begin{equation}
\begin{aligned}
P\left( \left\| \frac{1}{n}\sum_{i=1}^n Z_i \right\|_{L^2}  \ge
  \frac{4\epsilon}{5}\right)\le 2\exp \left \{ -\frac{16n\epsilon^2}{50K_1^2} \left(  1+
    \frac{1.62\epsilon}{\frac{1}{5}K_1}\right)^{-1} \right\}
\end{aligned}
\end{equation}
and for $\epsilon$ sufficient small that 
$\left(1+\frac{1.62\epsilon}{K_1/5}\right)\le 2$, one gets the desired result.
\end{proof}
\end{lemma}

Next we derive a similar result for the second derivative. Obtaining such a
result is a little more difficult because the operator norm of a linear operator
defined on a Hilbert space does not induce a Hilbert space structure. The
following discussion and intermediate results are useful to circumvent this problem.
\begin{definition}
Let $A: L^2 \to L^2$ be a linear operator. $A$ is said to be a Hilbert-Schmidt operator on $L^2$ if
\begin{equation}
\begin{aligned}
\|A\|^2_{HS} := \sum_{i=1}^{\infty} \|Ae_i\|^2_{L^2} < \infty
\end{aligned}
\end{equation}
where $\{e_i\}_{i=1}^{\infty}$ is an orthonormal basis of $L^2$.
\end{definition}
\begin{remark}
The above definition is independent of the choice of the orthonormal
basis. Furthermore, Hilbert-Schmidt operators form a Hilbert space 
with the following inner product: for two Hilbert-Schmidt operators $A$ and $B$,
the Hilbert-Schmidt inner product between $A$ and $B$ is defined as
\begin{equation}
\langle A,B\rangle_{HS}= \sum^{\infty}_{i=1}\langle Ae_i,Be_i\rangle_{L^2}
\end{equation}
where $\{e_1\}_{i=1}^{\infty}$ is any orthonormal basis of $L^2$.
Recall that the operator norm of bilinear operator $A$ is defined as 
\begin{equation}
\|A\|=\sup_{\{v\ : \ \|v\|_{L^2}=1\} }\|A(v)\|_{L^2}
\end{equation}
A standard result guarantees that $\|A\| \leq \|A\|_{HS}$.
\end{remark}
\begin{lemma}
\label{lemma:HSnormofB}
Let 
\begin{equation}
B( \cdot ,\cdot)  =  4K_h''(\|X -x\|^2_{L^2}) \langle x-X, \cdot \rangle_{L^2}\langle x-X, \cdot \rangle_{L^2}.
\end{equation}
Then $\|B\|_{HS} \leq 4K_2$ $P$-almost surely.
\begin{proof}
Let $Y= 2 \sqrt{ K''_h(\|X -x\|^2_{L^2})} (x-X)$, hence $Y\in L^2$. It is easily seen that $\|Y\|_{L^2} \leq 2\sqrt{K_2}$ $P$-almost
surely by (H1). Consider $\bar B(v) = \langle Y,v\rangle_{L^2} $ and let $\{e_i\}^{\infty}_{i=1}$ be an orthonormal basis.
We can write $Y =\sum_{i=1}^{\infty} y_i e_i$, where $y_i$ are random coefficients. Therefore,
$\|Y\|_{L^2}^2=\sum_{i=1}^{\infty} y_i^2 \leq 4 K_2$ $P$-almost surely. Finally,
$B : L^2 \to L^2$ can be expressed as
$B(v)= \bar B(v) Y$ and
\begin{equation}
\begin{aligned}
\|B\|_{HS}^2 &= \sum_{i=1}^{\infty} \| \bar B(e_i) Y\|^2_{L^2} =\sum_{i=1}^{\infty} \| \langle Y, e_i\rangle_{L^2} Y\|^2_{L^2} =
\sum_{i=1}^{\infty} \| y_iY\|^2_{L^2} \\
& = \|Y\|^2_{L^2} \sum_{i=1}^{\infty} y_i^2 =\|Y\|^4_{L^2}.
\end{aligned}
\end{equation}
This complete the proof.
\end{proof}
\end{lemma}

\begin{lemma}
\label{lemma:expInequalitySecondDerivative}
Under the same assumption of Lemma \ref{lemma:expInequality}, for $\epsilon$
small enough so that $\left(1+\frac{1.62\epsilon}{K_1/5}\right)\leq 2$, we have
\begin{equation}
\begin{aligned}
&P \left(\sup_{x\in B_{H^1_0} (0, M)} \| D^2\hat p_h(x) - D^2 p_h(x)\| \ge
  \epsilon\right) \\
& \le C \exp\left(-\frac{n\epsilon^2}{25K_2^2}+\frac{10MK_3}{\epsilon}\right ).
\end{aligned}
\end{equation}
\begin{proof}
%
% By lemma \ref{lemma:derivativesOfp} ,  $D^2E_P(\bar p_h(x)) = E_P(D^2 \bar
% p_h(x))$. Then as before, 
If $\epsilon'$ is taken to be $\epsilon / (10K_3)$, one has
\begin{equation}
\begin{aligned}
&P \left(\sup_{x\in B_{H^1_0} (0, M)} \| D^2\hat p_h(x) - D^2 p_h(x)\|
  \ge \epsilon\right) \\
&\le N_{\epsilon'}P \left( \| D^2\hat p_h(x_1) -
  D^2 p_h(x_1)\| \ge \frac{4\epsilon}{5}\right).
\end{aligned}
\end{equation}
Let
\begin{equation}
B_i( \cdot ,\cdot)  =  4K_h''(\|X_i -x\|^2_{L^2}) \langle x-X_i, \cdot
\rangle_{L^2}\langle x-X_i, \cdot \rangle_{L^2}
\end{equation}
and 
\begin{equation}
C_i( \cdot ,\cdot)  =  2K_h'(\|X_i -x\|^2_{L^2}) \langle \cdot , \cdot
\rangle_{L^2}.
\end{equation}
For any bilinear operator $T (v,w) = t\langle v ,w\rangle$, where $t\in \mathbb R$, then 
$\| T\| =|t|$.
Thus,
\begin{equation}
\begin{aligned}
&\| D^2\hat p_h(x) - D^2 p_h(x)\|  \\
&= \left \| \frac{1}{n}\sum_{i=1}^{n} \left(B_i -E_P(B_i) \right)  + \frac{1}{n}\sum_{i=1}^{n} \left(C_i -E_P(C_i)\right)\right\| \\
&\le  \left\| \frac{1}{n}\sum_{i=1}^{n} \left(B_i -E_P(B_i) \right)  \right\|+ \left \|\frac{1}{n}\sum_{i=1}^{n} \left(C_i -E_P(C_i)\right)\right\|\\
&\le \left \| \frac{1}{n}\sum_{i=1}^{n} \left(B_i -E_P(B_i) \right)  \right\| \\
&+2 \left|\frac{1}{n} \sum^{n}_{i=1}  K'_h(\|x-X_i\|^2_{L^2})  - E_PK'_h(\|x-X\|^2_{L^2}) \right|.
\end{aligned}
\end{equation}
As a result,
\begin{equation}
\begin{aligned}
&P \left( \| D^2\hat p_h(x_1) - p_h(x_1)\| \ge \frac{4\epsilon}{5}\right) \\
&\le P \left(  \left \| \sum_{i=1}^{n}\frac{1}{n} \left(B_i -E_p(B_i) \right)
  \right\| \ge \frac{2\epsilon}{5}\right) \\
&+ P\left(2 \left|\frac{1}{n}\sum^{n}_{i=1}  K'_h(\|x-X_i\|^2_{L^2})  -E_PK'_h(\|x-X\|^2_{L^2}) \right| \ge \frac{2\epsilon}{5}\right).
 \end{aligned}
\end{equation}
Since $|K'_h(\|x-X_i\|^2_{L^2})| \le K_2 $, Hoeffding's inequality
implies
\begin{equation}
\begin{aligned}
&P\left( \left|\frac{1}{n}\sum^{n}_{i=1}
    K'_h(\|x-X_i\|^2_{L^2})-E_PK'_h(\|x-X\|^2_{L^2}) \right| \ge
  \frac{2\epsilon}{5}\right) \\
% 2\exp\left( -\frac{n\epsilon^2}{200K_2^2}       \right).
& \le 2\exp\left( -\frac{8n\epsilon^2}{25K_2^2}       \right).
\end{aligned}
\end{equation}
In order to apply the Corollary of Lemma 4.3 in
\cite{yurinskiui1976exponential} on the Hilbert-Schmidt operator norm, it suffices to check that $B_i-E_P(B_i)$ has
bounded Hilbert-Schmidt norm. 
%and null expectation. 
Lemma \ref{lemma:HSnormofB} guarantees that
\begin{equation}
\label {eq:HScondition}
 \|B_i -E_P(B_i) \|_{HS} \le 8K_2
\end{equation}
almost surely.
 % and $E_P \left( B_i -E_P(B_i) \right) = 0$. Applying theorem ?? in Yurinskii, then 
Therefore, since $\| A\| \le \|A\|_{HS}$ for any bilinear operator $A$,  with small enough $\epsilon$, then
\begin{equation}
\begin{aligned}
& P \left(  \left \| \sum_{i=1}^{n}\frac{1}{n} \left(B_i -E_p(B_i) \right)  \right\|
  \ge \frac{2\epsilon}{5}\right) \\
  & \le P \left(  \left \| \sum_{i=1}^{n}\frac{1}{n} \left(B_i -E_p(B_i) \right)  \right\|_{HS}
  \ge \frac{2\epsilon}{5}\right) \\
& \le 2 \exp\left(  -\frac{4n\epsilon^2}{50K_2^2}
  \left(  1+ \frac{1.62\epsilon}{\frac{1}{5}K_1}\right)^{-1}\right) \\
& \le 2 \exp\left(  -\frac{n\epsilon^2}{25K_2^2}\right).
\end{aligned}
\end{equation}
\end{proof}
\end{lemma}

\bibliography{MyBibliography.bib}

\end{document}